\documentclass[a4paper,12pt]{amsart}

\usepackage[english]{babel}

\usepackage{lmodern}

\usepackage[utf8]{inputenc}

\usepackage[babel=true]{csquotes}

\usepackage{comment}

\usepackage{amsfonts}
\usepackage{amsthm} 
\usepackage{amsmath}
\usepackage{amssymb}
\usepackage{enumitem,hyperref}
\usepackage[dvips]{graphicx}

\usepackage{multirow} 
\usepackage{kbordermatrix}

\usepackage{brunnian}
\usepackage{tikz}

\usetikzlibrary{arrows}
\usetikzlibrary{matrix}
\usetikzlibrary{calc}

\theoremstyle{definition}
\newtheorem{defi}{Definition}[section] 
\newtheorem{definition}{Definition}[section] 

\theoremstyle{plain}
\newtheorem{lemma}[defi]{Lemma} 
\newtheorem{prop}[defi]{Proposition}
\newtheorem{proposition}[defi]{Proposition}
\newtheorem{conjecture}[defi]{Conjecture}
\newtheorem{theorem}[defi]{Theorem}
\newtheorem{corollary}[defi]{Corollary}

\newtheorem{question}[defi]{Question}

\theoremstyle{remark}
\newtheorem{remark}[defi]{Remark}
\newtheorem{exemple}[defi]{Example}

\newtheorem*{claim}{Claim}
\newtheorem*{step1}{Step 1}
\newtheorem*{step2}{Step 2}
\newtheorem*{step3}{Step 3}

\DeclareMathOperator{\id}{id}

\newcommand{\R}{\mathbb{R}}
\newcommand{\Q}{\mathbb{Q}}
\newcommand{\Z}{\mathbb{Z}}
\newcommand{\N}{\mathbb{N}}

\newcommand{\NN}{\mathcal{N}}

\newcommand{\smprod}[2]{\mbox{\footnotesize$\displaystyle\prod\limits_{#1}^{#2}$}}

\newcommand{\tmoplus}[2]{\mbox{$\textstyle \bigoplus\limits_{#1}^{#2}$}}

\newcommand{\lmfrac}[2]{\mbox{\small$\displaystyle\frac{#1}{#2}$}}
\newcommand{\smfrac}[2]{\mbox{\footnotesize$\displaystyle\frac{#1}{#2}$}} 
\newcommand{\tmfrac}[2]{\mbox{\large$\frac{#1}{#2}$}}

\newcommand{\tmsum}[2]{\mbox{$\textstyle \sum\limits_{#1}^{#2}$}}
\newcommand{\ttmsum}[2]{\mbox{\footnotesize{$\textstyle \sum\limits_{#1}^{#2}$}}} 
\def\doublesms{\setminus\hspace{-0.15cm}\setminus}

\def\bnm{\begin{enumerate}} 
\def\enm{\end{enumerate}}
\def\ba{\begin{array}} 
\def\ea{\end{array}} 
\def\bpp{\begin{pmatrix}}
\def\epp{\end{pmatrix}}

\def\ker{\operatorname{ker}}
 
\def\hom{\operatorname{Hom}}

\def\sms{\setminus}
\def\op{\operatorname}

\def\us{\underset}
\def\wti{\widetilde}
\def\what{\widehat}

\def\tautwo{\tau^{(2)}}
\def\vol{\op{vol}}
\def\JJ{\mathcal{J}}

\title[The leading coefficient of the $L^2$-Alexander torsion]{The leading coefficient of the $L^2$-Alexander torsion}
\author{Fathi Ben Aribi, Stefan Friedl and Gerrit Herrmann}

\address{Universit\'e de Gen\`eve, Section de math\'ematiques, 2-4 rue du Li\`evre, Case postale 64 1211 Gen\`eve 4, Suisse}
\email{fathi.benaribi@unige.ch}

\address{Fakult\"at f\"ur Mathematik\\ Universit\"at Regensburg\\ Germany}
\email{sfriedl@gmail.com}

\address{Fakult\"at f\"ur Mathematik\\ Universit\"at Regensburg\\ Germany}
\email{gerrit.herrmann@mathematik.uni-regensburg.de}

\begin{document}

\renewcommand{\proofname}{Proof}


\maketitle

\begin{abstract}
We give upper and lower bounds on the leading coefficients of the $L^2$-Alexander torsions of a $3$-manifold  $M$ in terms of hyperbolic volumes and of relative $L^2$-torsions of  sutured manifolds  obtained by cutting $M$ along certain surfaces.

We prove that for numerous families of knot exteriors the lower and upper bounds are equal, notably for exteriors of 2-bridge knots.  In particular we compute the leading coefficient explicitly for 2-bridge knots.
\end{abstract}

\section{Introduction}
We start out with introducing the following convention: throughout the paper we assume that all manifolds are  compact and oriented. 
By a hyperbolic 3-manifold we always mean a 
3-manifold with empty or toroidal boundary such that the interior admits a complete hyperbolic metric.

\subsection{Introduction to the $L^2$-Alexander torsion}
Let $N$ be  an irreducible 3-manifold with empty or toroidal boundary and let $\phi\in H^1(N;\R)$. 
The \emph{$L^2$-Alexander torsion $\tau^{(2)}(N,\phi)$} is a function $\R_{>0}\to \R_{\geqslant 0}$ that was introduced by Dubois, L\"uck and the second author~\cite{DFL16}. We will recall
the definition in Section~\ref{sec:prel}.
The $L^2$-Alexander torsion of $(N,\phi)$ is well-defined up to multiplication by a function of the form $t\mapsto t^k$ for some $k\in \R$. 
In the following, given two functions $f(t),g(t)\colon \R_{>0}\to \R$ we write $f(t)\doteq g(t)$ if there exists a $k\in \R$ such that $f(t)=t^k\cdot g(t)$ for all $t\in \R_{>0}$.

Perhaps the most interesting example is to consider a knot $K\subset S^3$. We denote by $E_K=S^3\sms \nu K$ the knot exterior, i.e.\ the complement of an open tubular neighborhood of $K$. Furthermore we denote by $\phi_K\in H^1(E_K;\Z)\cong \Z$ a generator. The $L^2$-Alexander torsion $\tau^{(2)}(K):=\tau^{(2)}(E_K,\phi_K)$ was initially introduced by Li-Zhang~\cite{LZ06} 
and has been known under the name of $L^2$-Alexander invariant (up to multiplication by a function of the form $t\mapsto \max\{1,t\}$).

From the definition using $L^2$-torsions the $L^2$-Alexander torsion might appear to be a rather mysterious invariant, but as is argued in \cite{DFL15b}, it can and should be viewed as a sibling of the more familiar twisted Alexander polynomials~\cite{Wa94,FV10} and of the higher-order Alexander polynomials~\cite{Co04}. 

Over the last few years the $L^2$-Alexander torsion has been the focus of intensive research.
In the following theorem we summarize some of the key results regarding the $L^2$-Alexander torsion.

\begin{theorem}\label{thm:previous}
Let $N$ be an irreducible 3-manifold  with empty or toroidal boundary and  $\phi\in H^1(N;\R)$. The following statements hold:
\bnm[font=\normalfont]
\item The evaluation of $\tautwo(N,\phi)$ at $t=1$ equals $\exp(\vol(N)/6\pi)$, where the volume $\vol(N)$ of $N$ is defined as the sum of the volumes of the hyperbolic pieces in the JSJ-decomposition of $N$.
\item If $N=N_1\sqcup N_2$ is the disjoint union of two 3-manifolds, then 
$$\tautwo(N,\phi)=\tautwo(N_1,\phi|_{N_1})\cdot \tautwo(N_2,\phi|_{N_2}).$$
\item 
If $N$ is obtained from a $($possibly disconnected$)$  3-manifold $M$ by gluing $M$ to itself via pairing up incompressible tori components of its boundary $\partial M$, then $\tautwo(N,\phi)=\tautwo(M,\phi|_{M})$.
\item The $L^2$-Alexander torsion $\tautwo(N,\phi)\colon \R_{>0}\to \R_{\geqslant 0}$ takes values in $\R_{>0}$.
\item The $L^2$-Alexander torsion $\tautwo(N,\phi)\colon \R_{>0}\to \R_{>0}$ is continuous.
\item If $\phi$ is rational, i.e.\ if $\phi\in H^1(N;\Q)$, then the $L^2$-Alexander torsion is symmetric in the sense that
\[ \hspace{1cm}\tautwo(N,\phi)(t) \,\,\doteq \,\,\tautwo(N,\phi)(t^{-1}).\]
\enm
\end{theorem} 

Here the first statement follows from the definitions and the work of L\"uck-Schick~\cite{LS99}. The second statement holds by definition. The proof of the third statement is basically identical to the proof of \cite[Theorem~5.5]{DFL16}.
The fourth statement was proved independently by Liu~\cite[Theorem~1.2]{Liu17} and L\"uck~\cite[Theorem~7.5]{Lu15}. The fifth statement was proved by Liu~\cite[Theorem~1.2]{Liu17}. Finally the last statement, which is a relatively straightforward consequence of Poincar\'e Duality, was proved by Dubois, L\"uck and the second author~\cite[Theorem~1.1]{DFL15a}. 
Note that we assume $N$ to be irreducible because the $L^2$-torsions are never defined for reducible 3-manifolds.

Given $(N,\phi)$ as above it is interesting to study the behavior of the $L^2$-Alexander torsion $\tau(N,\phi)(t)$ as $t\to \infty$. To formulate the known results we need to recall the definition of the Thurston norm of a connected irreducible 3-manifold $N$. 
Recall that for each $\phi\in H^1(N;\mathbb{Z})$ there is a properly embedded surface $\Sigma$ that represents $\phi$, via the Poincar\'e duality isomorphism 
$\op{PD}\colon H_2(N,\partial N;\Z) \to H^1(N;\mathbb{Z})$. Following~\cite{Th86} we define the \emph{Thurston norm} of a class $\phi\in H^1(N;\mathbb{Z})$ as 
 \[
x_N(\phi)\,:=\,\min \big\{ \chi_-(\Sigma) \,|\, \mbox{$\Sigma$ is a properly embedded surface with $\op{PD}([\Sigma]) = \phi$} \big\},
\]
where given a surface $\Sigma$ with components $\Sigma_1,\dots,\Sigma_k$ we define its \textit{complexity} as
\[\chi_-(\Sigma)\,\,:=\,\,\tmsum{i=1}{k} \max\{-\chi(\Sigma_i),0\}.\]
Thurston~\cite{Th86} showed that $x_N$ is indeed a seminorm on $H^1(N;\Z)$. It follows easily that $x_N$ can be extended to a seminorm on $H^1(N;\R)$ which we denote again by $x_N$. Two natural cases jump to mind:

\bnm
\item First of all, if $K\subset S^3$ is a non-trivial knot, then a straightforward argument implies that 
$x_{E_K}(\phi_K)=2g(K)-1$, where $g(K)$ denotes the minimal genus of a Seifert surface of $K$.
\item If $\phi\in H^1(N;\Z)$ is a fibered class, which means that there exists a surface bundle $p\colon N\to S^1$ such that $p_*=\phi\in \hom(\pi_1(N),\Z)\cong H^1(N;\Z)$, then  by \cite[Theorem~3]{Th86} we have $x_N(\phi)=\chi_-(F)$, where $F$ denotes the fiber of the surface bundle.
\enm

Now we can formulate the following theorem which supplements Theorem~\ref{thm:previous}. 

\begin{theorem}\label{thm:previous2}
Let $N\ne S^1\times D^2$ be a connected irreducible 3-manifold  with empty or toroidal boundary and let $\phi\in H^1(N;\R)$. The following statements hold:
\bnm[font=\normalfont]
\item If $N$ is a graph manifold, then $\tautwo(N,\phi)(t)\doteq \max\left \{1,t^{x_N(\phi)}\right \}$.
\item If $\phi$ is an integral fibered class, then there exists a $T\geqslant 1$ such that 
\[ \hspace{1cm} \tautwo(N,\phi)(t)\,\,\doteq\,\, \left\{ \ba{rl} t^{x_N(\phi)}, &\mbox{ if }t>T, \\ 1,&\mbox{ if }t<\frac{1}{T}.\ea\right.\]
In fact one can take $T$ to be the entropy of the monodromy of the fibration.
\item For any representative $\tau$ of $\tautwo(N,\phi)$ we have
\[ \hspace{1cm} \lim_{t \to \infty} \smfrac{\ln(\tau(t))}{\ln(t)}-\lim_{t \to 0^+} \smfrac{\ln(\tau(t))}{\ln(t)}\,\,=\,\, x_N(\phi). \]
In particular both limits on the left hand side exist.
\item There exists a $C(N,\phi)\in \R_{>0}$ such that for any representative $\tau$ of $\tautwo(N,\phi)$ there exists a $k\in \R$ with
\[ \hspace{1cm} \lim_{t \to \infty} \smfrac{C(N,\phi)\cdot t^k}{\tau(t)}\,\,=\,\,1.\]
\enm
\end{theorem}

The first statement was proved by the third author~\cite[Corollary~1.2]{Her16}, extending earlier work of Dubois-Wegner~\cite{DW15} and the first author~\cite{BA16a}. The second statement is proved in~\cite[Theorem~1.3]{DFL15a}. The third statement was proved by Liu~\cite[Theorem~1.2]{Liu17}. For rational $\phi$, i.e.\ for $\phi\in H^1(N;\Q)$, the statement was independently obtained by the second author and L\"uck~\cite[Theorem~0.1]{FL15}. Finally the last statement is again due to Liu~\cite[Theorem~1.2]{Liu17}. Both proofs of the third statement relies in both cases on the work of Agol~\cite{Ag08,Ag13}, Przytycki-Wise~\cite{PW12} and Wise~\cite{Wi12}.

\bigskip

\begin{definition}
Let $N$ be a connected irreducible 3-manifold  with empty or toroidal boundary and let $\phi\in H^1(N;\R)$. The number $C(N,\phi)\in \R_{>0}$ of Theorem \ref{thm:previous2} (4) will be referred as  the \emph{leading coefficient of $\tautwo(N,\phi)$}.
\end{definition}

The following proposition lists several properties of the leading coefficient, following from the definitions or from the work of \cite{Liu17}.

\begin{proposition}\label{prop:leading}
	Let $N$ be a connected irreducible 3-manifold  with empty or toroidal boundary and let $\phi\in H^1(N;\R)$.
	\bnm[font=\normalfont]
	\item If $\phi=0$ then $C(N,\phi)=\tautwo(N)$.
	\item If $\phi$ is a fibered class then $C(N,\phi)=1$.
	\item For any $r \in \R^*$, $C(N,r \phi) = C(N,\phi)$.
	\item The function $H^1(N;\R)\to \R$ given by $\phi\mapsto C(N,\phi)$ is upper semicontinuous.
	\item The leading coefficient $C(N,\phi)$ lies in the interval $\left [1,e^{\vol(N)/6\pi}\right ]$.
	\enm
\end{proposition}

Proposition \ref{prop:leading} (2) may remind the reader of the similar monicity of the Alexander polynomial of a fibered knot.
Note also that Proposition \ref{prop:leading} (1), (3) and (4) (respectively value at $0$, constancy on rays and upper semi-continuity) together with Theorem \ref{thm:previous} (1) imply the upper bound of Proposition \ref{prop:leading} (5).
Finally we point out that Liu~\cite[Chapter~9]{Liu17} shows that the function $\phi\mapsto C(N,\phi)$ is in general not continuous.

\bigskip

\subsection{The main results}
In this paper we are mostly concerned with the following question.

\begin{question}
Given a connected irreducible 3-manifold $N$  with empty or toroidal boundary and $\phi\in H^1(N;\Z)$, how can we express $C(N,\phi)$ in terms of the topology and geometry of $N$?
\end{question}

Before we can state our main theorem we need to introduce some notation.
Let $\Sigma$ be a properly embedded surface in an irreducible 3-manifold $N$.
\bnm
\item We say that $\Sigma$ is \emph{Thurston norm minimizing} if $x_N([\Sigma])=\chi_-(\Sigma)$ and if no component of $\Sigma$ is a sphere, a disk, a compressible torus or a boundary parallel annulus. 
If $N\ne S^1\times D^2$, then it follows  from standard arguments and from our hypothesis that $N$ is irreducible that any $\phi$ can be represented by a Thurston norm minimizing surface. 
Indeed, any component that is a compressible torus or a boundary parallel annulus is always null homologous, and spheres and disks are null homologous as well when $N$ is irreducible and not $S^1\times D^2$.
Note that the empty surface is the unique Thurston norm minimizing surface representing the trivial homology class.
\item We denote by $\Sigma\times [-1,1]$ a closed tubular neighborhood of $\Sigma$. Furthermore  we write $N\doublesms \Sigma:=N\sms \Sigma \times (-1,1)$ and we write $\Sigma_\pm :=\Sigma \times \{\pm 1\}$. 
\item We denote by $\JJ(N)$ the set of JSJ-components of $N$.
\enm

The following theorem is proved in~\cite{Her18}.

\begin{theorem}\label{thm:reltorexist}
Let $N\ne S^1\times D^2$ be a connected irreducible $3$-manifold with empty or toroidal boundary
and let $\Sigma$ be a properly embedded surface in $N$. If $\Sigma$ is Thurston norm minimizing in $N$, then
the  $L^2$-Betti numbers of the pair $(N\doublesms \Sigma,\Sigma_-)$ vanish and  the relative $L^2$-torsion $\tautwo(N\doublesms \Sigma,\Sigma_-)\in \R_{>0}$ is defined. 
\end{theorem}

See Section \ref{sec:prel} for a precise definition of the term $\tautwo(N\doublesms \Sigma,\Sigma_-)$. The following is now our main theorem.

\begin{theorem}\label{main-theorem}
Let $N$ be a connected irreducible $3$-manifold with empty or toroidal boundary.
Furthermore let $\phi \in H^1(N;\Z)$ 
and let $\Sigma$ be a Thurston norm minimizing surface dual to $\phi$. Then the following two inequalities hold: 
\[ \smprod{{\ba{c}{\tiny{\mbox{$M\in \JJ(N)$}}}\\ \tiny{\mbox{with $\phi|_M=0$}}\ea}}{} e^{\vol(M)/6\pi}
\,\, \leqslant\,\,
C(N,\phi)
\,\, \leqslant\,\,
\tau^{(2)}(N\doublesms \Sigma, \Sigma_-).\]
\end{theorem}

It is natural to ask for which cases the inequalities of Theorem~\ref{main-theorem} are in fact equalities. In Section~\ref{sec:ahyperbolic} we define the class of ahyperbolic surfaces. The precise definition is irrelevant at the moment, but in Proposition~\ref{prop:torsionahyperbolic} we show that for an ahyperbolic surface all three terms in Theorem~\ref{main-theorem} are in fact equal to 1.

It turns out that all Thurston norm minimizing surfaces in a graph manifold are ahyperbolic. More interesting examples are given by Agol-Dunfield \cite{AD15}, who showed that all 2-bridge knots admit an ahyperbolic surface. Since there exists a family of non-fibered hyperbolic 2-bridge knots, namely the family of twist knots, we have the following corollary.\\

\noindent \textbf{Corollary \ref{cor:nonfiberedhyperbolic}.} \emph{There exist infinitely many \emph{non-fibered} hyperbolic knots $K$ in $S^3$ such that 
\[ C(E_K,\phi_K)\,\,=\,\,1.\]}

Note that we already knew that $ C(E_K,\phi_K)$ was equal to  $1$ in the case $K$ was fibered, thanks to Theorem \ref{thm:previous2} (2).

At this point it is natural to wonder if $1$ is the only possible value for the leading coefficient of a knot. We  answer in the negative:\\

\noindent \textbf{Corollary \ref{cor:>1}}  \emph{
The set of leading coefficients $C(E_K,\phi_K)$ (where the index $K$ runs over the set of all knots) is infinite.}

\emph{Furthermore, this set of leading coefficients contains a subset which is bijective to the set of hyperbolic volumes $\vol(E_K)$ of hyperbolic knots.}

\

We construct such examples of knots, with leading coefficient greater than $1$, 
in Section \ref{section:examples} as Whitehead doubles of hyperbolic knots, and we can compute the exact value of the leading coefficient for these examples. In particular the knots we provide to prove Corollary \ref{cor:>1} are non-hyperbolic.
We propose the following conjecture.

\begin{conjecture}\label{conjecture}\mbox{}
	\bnm[font=\normalfont]
	\item For every irreducible $3$-manifold $N$ with empty or toroidal boundary, any class $\phi \in H^1(N;\Z)$ 
	and any two Thurston norm minimizing surfaces $\Sigma, \Sigma'$ dual to $\phi$, we have
	\[ \hspace{1cm} \tautwo(N\doublesms \Sigma,\Sigma_-) \,\,=\,\, \tautwo(N\doublesms \Sigma',\Sigma'_-).\]
	\item For every irreducible 3-manifold $N$  with empty or toroidal boundary and any class $\phi\in H^1(N;\Z)$ the second inequality of Theorem~\ref{main-theorem} is an equality, i.e.\ we conjecture that
	for any Thurston norm minimizing surface $\Sigma$ dual to $\phi$ we have
	\[ \hspace{1cm} C(N,\phi) \,\,=\,\, \tautwo(N\doublesms \Sigma,\Sigma_-).\]
	\item There exists a hyperbolic 3-manifold $N$  with empty or toroidal boundary and a class $\phi \in H^1(N;\Z) \setminus \{ 0 \}$ such that $C(N,\phi)>1$.
	\enm 
\end{conjecture}

Note that part (2) of Conjecture \ref{conjecture} would immediately imply part (1). 
A different way of formulating Conjecture~\ref{conjecture}  (1) is to say that we conjecture that the first inequality of Theorem~\ref{main-theorem} is in general not an equality.
Furthermore Conjecture~\ref{conjecture} (2) says that  we expect that the term on the right-hand side is independent of the choice of $\Sigma$ and that in fact the second inequality of Theorem~\ref{main-theorem} is an equality. 

Conjecture~\ref{conjecture} (2) is motivated by the case of the classical Alexander polynomial. Indeed, for any knot $K$  whose Alexander polynomial $\Delta_K$ is of maximal degree $2g(K)$ and for any $\Sigma$ a minimal genus Seifert surface of $K$, the leading coefficient of $\Delta_K$ is equal to the order of $H_1(X_K\doublesms \Sigma,\Sigma_-)$.
Conjecture~\ref{conjecture} (3) would follow from (2), from an expected (but only conjectured at the moment) generalization of Theorem \ref{thm:previous} (1) to pairs 
$(N\doublesms \Sigma,\Sigma_-)$ and from the fact that there exists examples of such pairs whose volume is strictly greater than the volume in the first term of Theorem \ref{main-theorem}.

Note that Corollary \ref{cor:nonfiberedhyperbolic} implies, perhaps somewhat disappointingly, that in general the leading coefficient does not detect fiberedness of hyperbolic 3-manifolds. Nonetheless it is an interesting question whether the $L^2$-Alexander torsion 
detects fiberedness of hyperbolic 3-manifolds.

\subsection{The $L^2$-Alexander torsion and quantum invariants}
\ 
As a generalization of the volume conjecture \cite{MM01}, one can wonder if and how we can expect the $L^2$-Alexander torsions to be approximated by quantum invariants. As such, knowing values such as the leading coefficient can help test the plausibility of such conjectures.
As an example, let us consider the following conjecture, formulated by Xiao-Song Lin \cite[p.~9]{Lin05} in 2005.

\begin{conjecture} \textup{(Xiao-Song Lin 2005)}
Let $K\subset S^3$ be a knot. For every $N\in \N$ we denote by $J_K(N,x)\in \Z[q^{\pm 1}]$ the normalized $N$-th colored Jones polynomial \textup{(}We refer to \cite{MM01} and \cite{Mur11} for the definition.$)$ For every $t\in  \mathbb{C}^*$ the following equality holds:
\begin{equation} \us{N\to \infty}{\op{lim}} \Big| J_K\Big(N,\exp\Big(\tmfrac{2\pi it}{N}\Big)\Big)\Big|^{\frac{1}{3N}}\,\,\dot{=}\,\,\tau^{(2)}(K)(t)\cdot \max\{1,|t|\}.\label{qu:lin}\end{equation}
Xiao-Song Lin adds the comment that ``the final form of this conjecture is subject to modification''.
\end{conjecture}

The motivation for the conjecture surely stems from the fact that for $t=1$ the above question is equivalent to the volume conjecture \cite{MM01}.
There are at least two reasons why (\ref{qu:lin}) cannot hold as stated:
\bnm
\item[(a)] Morton and Traczyk \cite{MT88} showed that colored Jones polynomials are invariant under mutation. On the other hand $\tau^{(2)}(K)$ is not invariant under mutation. This can be seen as follows: 
by Theorem~\ref{thm:previous2} the invariant $\tau^{(2)}(K)$ detects the genus, but the genus is not a mutation invariant. In fact the 
Conway knot and the Kinoshita-Terasaka knot are mutants, but their genera are respectively $2$ and $3$.
\item[(b)] It follows from \cite[Theorem 1]{GL05} that for any knot $K$ there exists an $r>0$ such that for any $t\in (0,r)$ the left hand side of (\ref{qu:lin}) converges to $1$.
Were the conjecture true, the right hand side would be $1$ too, and the leading coefficient as well; however Corollary~\ref{cor:>1}  implies that this cannot hold in general. 
\enm

We take the freedom to rephrase Lin's question as follows:

\begin{question}
Is the $L^2$-Alexander invariant $\tau^{(2)}(K)$ determined by quantum invariants?
\end{question}

A rather speculative idea is that perhaps the results of Futer, Kalfagianni and Purcell \cite{FKP13} can be used to build a bridge between quantum invariants and the $L^2$-Alexander invariant. 

\subsection*{Organization}
This paper is organized as follows. In Section \ref{sec:prel} we recall the definitions of the $L^2$-torsion and of the $L^2$-Alexander torsion. In Section \ref{sec:turaev} we apply Turaev's algorithm on embedded surfaces to compare relative $L^2$-torsions.
In Section \ref{sec:proof} we prove the main Theorem \ref{main-theorem}. In Section \ref{sec:ahyperbolic} we introduce ahyperbolic surfaces and study their corresponding relative $L^2$-torsions. Finally in Section \ref{section:examples} we compute the leading coefficient for the Borromean rings and prove Corollary \ref{cor:>1}.

\subsection*{Conventions.} 
As mentioned in the beginning, throughout the paper, unless we say explicitly otherwise, we assume that all manifolds are  compact and oriented.
Furthermore 
all groups are understood to be countable.

\subsection*{Acknowledgements.}
The first author was supported by the Swiss National Science Foundation, subsidy 200021$\_$162431, at the Universit\'e de Gen\`eve.
The second and the third author gratefully acknowledge the
support provided by the SFB 1085 `Higher Invariants' at the University of Regensburg,
funded by the Deutsche Forschungsgemeinschaft DFG. 
We thank the referees for their many helpful comments and suggestions.

\section{Preliminaries}\label{sec:prel}
In this section, for the most part we follow \cite{Lu02} and \cite{DFL16}. We refer to these references for more details.


\subsection{The von Neumann dimension}
\label{sub:dimension}

Given a group~$G$, the completion of the algebra~$ \mathbb{C}[G]$ endowed with the scalar product
$ \left \langle \sum_{g \in G} \lambda_g g , \sum_{g \in G} \mu_g g \right \rangle:= \sum_{g \in G} \lambda_g \overline{\mu_g}$
is the Hilbert space 
$$ \ell^2(G)\,\,:=\,\, \Big \{ \,\tmsum{g \in G}{} \lambda_g g \ \Big| \ \lambda_g \in \mathbb{C} , \tmsum{g \in G}{} | \lambda_g |^2 < \infty  \Big\}$$ of square-summable complex functions on~$G$. We denote by~$B(\ell^2(G))$ the algebra of operators on~$\ell^2(G)$ that are bounded with respect to the operator norm. 

Given~$h \in G$, we define the corresponding \textit{left-} and \textit{right-multiplication operators} $L_{h}$ and $R_{h}$ in $B(\ell^2(G))$ as extensions of the classical automorphisms of $G$ $(g \mapsto hg)$ and $(g \mapsto gh)$.
One can extend the operators $R_{h}$ $\mathbb{C}$-linearly to an operator $R_w\colon \ell^2(G) \to \ell^2(G)$ for any $w \in  \mathbb{C}[G]$. Moreover, if~$\ell^2(G)^n$ is endowed with its usual Hilbert space structure and~$A = \left ( a_{i,j} \right ) \in M_{p,q}( \mathbb{C}[G])$ is a~$ \mathbb{C}[G]$-valued~$p\times q$ matrix, then the right multiplication by
$$R_A:= \left (R_{a_{i,j}}\right )_{1 \leqslant i \leqslant p, 1 \leqslant j \leqslant q}$$ provides a bounded operator~$\ell^2(G)^{p} \rightarrow \ell^2(G)^{q}$. Note that here we consider the vectors of $\ell^2(G)^{p}$ as row vectors and the ``matrix operator'' $R_A$ acts  on the right; notably one gets $R_{A B} = R_B \circ R_A$. In most cases, when there is no danger of confusion, given $A\in M_{p,q}(\mathbb{C}[G])$ we denote by $A$ also the corresponding operator, i.e.\ we just write $A$ instead of $R_A$.

The \textit{von Neumann algebra}~$\mathcal{N}(G)$ of the group~$G$ is the sub-algebra of $B(\ell^2(G))$ made up of~$G$-equivariant operators (i.e. operators that commute with all left multiplications $L_h$).  A \textit{finitely generated Hilbert~$\mathcal{N}(G)$-module} consists of a Hilbert space~$V$ together with a left~$G$-action by isometries such that there exists a positive integer~$m$ and a $G$-equivariant embedding~$\varphi$ of~$V$ into~$\bigoplus_{i=1}^m \ell^2(G)$. A \textit{morphism of finitely generated Hilbert~$\mathcal{N}(G)$-modules} $f \colon U \rightarrow V$ is a linear bounded map which is $G$-equivariant.

Denoting by $e$ the neutral element of~$G$, the von Neumann algebra of~$G$ is endowed with the \textit{trace}~$\mathrm{tr}_{\mathcal{N}(G)} \colon \mathcal{N}(G) \rightarrow \mathbb{C}, \phi \mapsto \left \langle \phi (e) , e \right \rangle$ which extends to~$\mathrm{tr}_{\mathcal{N}(G)} \colon M_{n,n}(\mathcal{N}(G)) \rightarrow  \mathbb{C}$ by summing up the traces of the diagonal elements.

\begin{definition}
The \textit{von Neumann dimension} of a finitely generated Hilbert~$\mathcal{N}(G)$-module $V$ is defined as
~$$\dim_{\mathcal{N}(G)}(V) := \mathrm{tr}_{\mathcal{N}(G)}\left (\mathrm{pr}_{\varphi(V)}\right ) \in \R_{\geqslant 0},$$
where~$ \mathrm{pr}_{\varphi(V)} \colon \bigoplus_{i=1}^m \ell^2(G) \to \bigoplus_{i=1}^m \ell^2(G)~$ is the orthogonal projection onto~$\varphi(V)$.
\end{definition}

By \cite[Chapter~1.1.3]{Lu02} the von Neumann dimension does not depend on the embedding of~$V$ into the finite direct sum of copies of~$\ell^2(G)$.

\subsection{The Fuglede-Kadison determinant}
\label{sub:Fuglede}

The \textit{spectral density} $F(f)\colon \mathbb{R}_{\geqslant 0}\to \mathbb{R}_{\geqslant 0}$ of a morphism~$f \colon U \to V$ of finitely generated Hilbert~$\mathcal{N}(G)$-modules is defined as the map that sends~$\lambda \in \mathbb{R}_{\geqslant 0}$ to 
$$ F(f)(\lambda):= \sup \{
 \dim_{\mathcal{N}(G)} (L) | L \in \mathcal{L}(f,\lambda) \},$$
where~$\mathcal{L}(f,\lambda)$ is the set of finitely generated Hilbert~$\mathcal{N}(G)$-submodules of~$U$ on which the restriction of~$f$ has a norm smaller or equal to~$\lambda$. Since~$F(f)(\lambda)$ is monotonous and right-continuous it defines a measure~$dF(f)$ on the Borel set of~$\R_{\geqslant 0}$ that is uniquely determined by the equation~$dF(f)((a,b]) = F(f)(b)-F(f)(a)$ for all~$a<b$.

\begin{definition} \label{def detFK}
The \textit{Fuglede-Kadison determinant of~$f$} is defined by
\begin{equation*}\label{detFK}
{\det}_{G}(f)=
\begin{cases}
 \exp \left ( \int_{0^+}^\infty \ln(\lambda) \, dF(f)(\lambda) \right ) & \mbox{if } \int_{0^+}^\infty \ln(\lambda) \, dF(f)(\lambda) > -\infty, \\ 
0 & \mbox{otherwise.}
 \end{cases} 
\end{equation*}
Moreover, when~$\int_{0^+}^\infty \ln(\lambda) \, dF(f)(\lambda) > -\infty$, one says that \emph{$f$ is of determinant class}. 

If $A\in M_{n,n}( \mathbb{C}[G])$ then we define the $\emph{regular Fuglede-Kadison determinant}$ of $A$ by

\begin{equation*}\label{detrFK}
{\det}^r_{G}(A)=
\begin{cases}
{\det}_G (R_A) & \mbox{if $R_A$ is injective and of determinant class,} \\ 
0 & \mbox{otherwise.}
\end{cases} 
\end{equation*}
\end{definition}

The following proposition lists some basic properties of the regular Fuglede-Kadison determinant that will be used in later computations. The proposition follows easily from the results in  \cite[Section 3.2]{Lu02}.

\begin{proposition}\label{prop:op_det} 
Let $G$ be a group, and $n,p \in \Z_{>0}$. Then:
\bnm[font=\normalfont]
\item For all $\lambda \in \mathbb{C}$, $g\in G$, one has 
${\det}^r_{G}(\lambda g) = |\lambda|$.
\item For all $A,B \in M_{n,n}( \mathbb{C}[G])$, one has ${\det}^r_{G}(A B) = {\det}^r_{G}(A) {\det}^r_{G}(B)$.
\item For all $A \in M_{n,n}( \mathbb{C}[G]), C \in M_{n,p}( \mathbb{C}[G]), D \in M_{p,p}( \mathbb{C}[G])$, one has 
$${\det}^r_{G} \begin{pmatrix}
A & C \\ 0 & D
\end{pmatrix} = {\det}^r_{G}(A) {\det}^r_{G}(D).$$
\item Taking the transpose or permuting rows or columns leaves ${\det}^r_{G}$ unchanged.
\item For any group inclusion $i\colon H \hookrightarrow G$, and any $E \in M_{n,n}( \mathbb{C}[H])$, one has
$$ {\det}^r_{G} (i(E)) = {\det}^r_{H} (E)
.$$
\item Let $g\in G$ be an element of infinite order. Then for any $t\in \mathbb{C}$ we have
\[ {\det}^r_G(1 - t \cdot g) \,\,=\,\, \max\{1,|t|\}.\]
\enm
\end{proposition}

The regular Fuglede-Kadison determinant sometimes behaves better than the usual Fuglede-Kadison determinant. For example we will make use of the following fact recently proven by Liu \cite[Lemma 3.1]{Liu17}.

\begin{lemma}\label{lem:uppersemicont}
	Let $A_k\in M_{n,n}( \mathbb{C}[G]), k\in \N$, be a sequence converging to some $A\in M_{n, n}( \mathbb{C}[G])$
	in the norm topology, then
	\[ \limsup_{k\to\infty} {\det}^r_{G}(A_k) \,\leqslant \,{\det}^r_{G}(A). \]
\end{lemma}

\subsection{$L^2$-torsions}\label{section:def-l2-torsion}

A \textit{finite Hilbert $\NN(G)$-chain complex} $C_*$ is a sequence of morphisms of finitely generated Hilbert $\NN(G)$-modules
$$C_* = 0 \to C_n \overset{\partial_n}{\longrightarrow} C_{n-1} 
\overset{\partial_{n-1}}{\longrightarrow} \ldots
\overset{\partial_2}{\longrightarrow} C_1 \overset{\partial_1}{\longrightarrow} C_0 \to 0$$
such that $\partial_p \circ \partial_{p+1} =0$ for all $p$.
The \textit{$p$-th~$L^2$-homology} of such a chain complex~$C_*$ is the finitely generated Hilbert~$\NN(G)$-module
$$H_p^{(2)}(C_*) := \textrm{Ker}(\partial_p) / \overline{\textrm{Im}(\partial_{p+1})}$$
obtained by quotienting by the \textit{closure} of the image of $\partial_{p+1}$.
The \textit{$p$-th $L^2$-Betti number of~$C_*$} is defined as~$b_p^{(2)}(C_*) := \dim_{\NN(G)}(H_p^{(2)}(C_*))$. A finite Hilbert $\NN(G)$-chain complex~$C_*$ is \textit{weakly acyclic} if its~$L^2$-homology is trivial (i.e. if all its~$L^2$-Betti numbers vanish) and of \textit{determinant class} if all the operators~$\partial_p$ are of determinant class. 

\begin{definition}
If~$C_*$ is weakly acyclic and of determinant class, define its \textit{$L^2$-torsion} by
$$\tautwo(C_*) := \prod_{i=1}^n \det {}_{G}(\partial_i)^{(-1)^i} \in \R_{>0},$$
and set~$\tautwo(C_*)=0$ otherwise.
\end{definition}

Let~$X$ be a compact connected CW-complex endowed with a base point~$z$ and let~$Y$ be a CW-subcomplex of~$X$. We write $G=\pi_1(X,z)$, we denote by~$p\colon \tilde{X} \rightarrow X$ the universal cover of~$X$ and we write~$\tilde{Y}=p^{-1}(Y)$. The natural left action of~$G = \pi_1(X,z)$ on~$\tilde{X}$ gives rise to a left~$\mathbb{Z}[G]$-module structure on the cellular chain complex~$C_*\left (\tilde{X},\tilde{Y}\right )$. 
By picking a lift of each cell of $X\sms Y$ to $\tilde{X}\sms \tilde{Y}$ we can view 
$C_*\left (\tilde{X},\tilde{Y}\right )$ as a \emph{based} free $\Z[G]$-chain complex.

Now suppose we are given a homomorphism~$\phi \colon G \rightarrow \mathbb{R}$ and some $t>0$. We denote by
$\kappa(G, \phi,t) \colon \mathbb{Z}[G] \rightarrow \mathbb{R}[G]$
the ring homomorphism~$g \mapsto t^{\phi(g)} g$.
There is a right action of~$G$ on~$\ell^2(G)$ given by~$a \cdot g= R_{\kappa(G, \phi,t)(g)} (a)$ where~$a \in \ell^2(G)$ and~$g \in G$; this turns~$\ell^2(G)$ into a right~$\mathbb{Z}[G]$-module.
The \textit{$\mathcal{N}(G)$-cellular chain complex} of the pair~$(X,Y)$ associated to~$(\phi,t)$ is the finite Hilbert~$\NN(G)$-chain complex
$$ C_*^{(2)}(X,Y, \phi,t):=\ell^2(G) \otimes_{\mathbb{Z}[G]}C_*\left (\tilde{X},\tilde{Y}\right ),$$
and the \textit{$L^2$-homology} of~$(X,Y)$ associated to~$(\phi,t)$, denoted ~$H_*^{(2)}(X,Y,\phi,t)$, is obtained by taking the $L^2$-homology of $C_*^{(2)}(X,Y, \phi,t)$.

We define the \textit{$L^2$-Alexander torsion} of $(X,Y,\phi)$ at $t>0$ to be
\[ 
 \tautwo(X,Y,\phi)(t)\,\, := \,\, \left\{ \ba{ll} \tautwo\left (C_*^{(2)}(X,Y,\phi,t)\right ),&
 \mbox{ if $C_*^{(2)}(X,Y,\phi,t)$ is weakly acyclic and}\\
 &\quad \mbox{of determinant class},\\
 0,&\mbox{otherwise.}\ea\right.\]
As is explained in \cite[Lemma~4.1]{DFL16} the function $t\mapsto \tautwo(X,Y,\phi)(t)$, up to multiplication by a function of the form $t\mapsto t^k$ for some fixed $k\in \R$, does not depend on the choice of the lift of the cells of $X\sms Y$.

When $Y= \emptyset$, we write 
$C_*^{(2)}(X, \phi,t)$ instead of 
$C_*^{(2)}(X,Y, \phi,t)$
and
$\tautwo(X,\phi)$ instead of
$\tautwo(X,Y,\phi)$.

When $\phi$ is the zero map, $t$ becomes irrelevant and we write 
$C_*^{(2)}(X,Y)$ instead of
$C_*^{(2)}(X,Y, \phi,t)$ and
$\tautwo(X,Y)$ instead of $\tautwo(X,Y,\phi)$.
We call $\tautwo(X,Y)$ the\textit{ relative $L^2$-torsion} of $(X,Y)$ and 
$\tautwo(X) = \tautwo(X,\emptyset)$ the \textit{$L^2$-torsion} of $X$. 

Given a  connected manifold $M$ and a  submanifold $N$ we can use triangulations to view the pair $(M,N)$ as a pair of CW-complexes. (Recall that all manifolds are assumed to be compact.) As is discussed in \cite[p.~160]{Lu02}, the corresponding $L^2$-torsions do not depend on the choice of triangulation. Alternatively, if $M$ is a 3-manifold one can also use that the Whitehead group of $\pi_1(M)$ is trivial, see e.g.\ \cite[(C.36)]{AFW15}.

We define the $L^2$-torsion of a disconnected 3-manifold pair as the product of the $L^2$-torsions of the components.

The following lemma is proved for the case $N=\emptyset$ in \cite[Theorem~3.93]{Lu02}. The proofs carry over without any changes to the relative case.

\begin{lemma}\label{lem:tautwo-basics}
Let $(M,N)$ be a pair of  manifolds such that $M$ is connected and such that $\pi_1(M)$ is residually finite.
$($Note that the fundamental group of any compact 3-manifold is residually finite, see \cite{Hem87,AFW15}.$)$
\bnm[font=\normalfont]
\item If $M=N\times [0,1]$, then for any $s\in [0,1]$ we have $\tautwo(M,N\times \{s\})=1$.
\item 
Suppose that the $L^2$-Betti numbers of $(M,N)$ vanish. 
Let $p\colon \wti{M}\to M$ be a finite covering. We write $\wti{N}:=p^{-1}(N)$. Then 
the $L^2$-Betti numbers of $(M,N)$ also vanish and we have
\[ \tautwo\left (\wti{M},\wti{N}\right )\,\,=\,\, \tautwo(M,N)^{\left [\wti{M}:M\right ]}.\]
\item If $M$ is an $S^1$-bundle over a manifold $X$, e.g.\ if $M=S^1\times X$ for some manifold $X$, then the $L^2$-Betti numbers of $M$ vanish and $\tautwo(M)=1$. 
\item If the $L^2$-Betti numbers of $M$ and $N$ vanish, e.g.\ if $M$ and $N$ are $S^1$-bundles and if $\pi_1(N)\to \pi_1(M)$ is a monomorphism, then
\[ \tautwo(M)\,\,=\,\,\tautwo(M,N)\cdot \tautwo(N).\]
\item 
Suppose that the $L^2$-Betti numbers of $(M,N)$ vanish. Furthermore suppose that $(M,N)=(X\cup Y,C\cup D)$ where $X$ and $Y$ are submanifolds such that each component of $X\cap Y$ is a submanifold of $\partial X$ and $\partial Y$ and the same holds for $(M,X,Y)$ replaced by $(N,C,D)$.
If for each component $Z$ of $X\cap Y$ the $L^2$-Betti numbers of $(Z,Z\cap N)$ vanish and if the induced maps $\pi_1(Z)\to \pi_1(X)$ and $\pi_1(Z)\to \pi_1(Y)$ are monomorphisms, then
\[ \hspace{1cm} \tautwo(M,N)=\tautwo(X,C)\cdot\tautwo(Y,D)\cdot \tautwo(X\cap Y,C\cap D)^{-1}.\]
\enm
\end{lemma}

We make the following trivial observation which follows immediately from the definitions.

\begin{lemma}\label{lem:zero-phi}
Let $N$ be a  manifold.
\bnm[font=\normalfont]
\item For any $\phi\in H^1(N;\R)$ we have $\tautwo(N,\phi)(t=1)=\tautwo(N)$.
\item If $\phi\in H^1(N;\R)$ is the zero class, then $\tautwo(N,\phi)$ is a constant map, in particular $\tautwo(N,\phi)(t)=\tautwo(N)$ for all $t\in \R_{>0}$.
\item For any $\phi\in H^1(N;\R)$ and any $r\in \R$ we have $\tautwo(N,r \phi)(t)=\tautwo(N,\phi)(t^r)$.
\enm
\end{lemma}

We end this section with recalling the following lemma which allows one to calculate the torsion for a chain complex of a 3-manifold. In this way it was first stated in \cite[Lemma 3.2]{DFL16} but the ideas go back to \cite[Theorem 2.2]{Tu01}.

\begin{lemma}\label{lem:calculatetorsion}
 Let $G$ be a group, $j,k,l$ integers such that $j< k$ and $A,B,C$ matrices with entries in $ \mathbb{C}[G]$ of the respective sizes $(k+l-j)\times l$, $k \times (k+l-j)$ and $j \times k$. We consider the complex
 \[
 	C_*\colon 0 \longrightarrow \ell^2(G)^j \overset{R_C}{\longrightarrow} \ell^2(G)^k \overset{R_B}{\longrightarrow} \ell^2(G)^{k+l-j} \overset{R_A}{\longrightarrow} \ell^2(G)^l \longrightarrow 0. 
 \]

Let $L\subset\left\{1,\ldots,k+l-j \right\}$ be a subset of size $l$ and $J\subset\left\{1,\ldots k \right\}$ a subset of size $j$. We write 
 \begin{align*}
 	A(L):= &\mbox{ rows in A corresponding to } L, \\
 	B(J,L):= &\mbox{ result of deleting the columns of $B$ corresponding to $L$}\\
 	&\mbox{ and deleting the rows corresponding to J},\\
 	C(J):= &\mbox{ columns of C corresponding to $J$.}
 \end{align*}
 If ${\det}^r_{G}(A(L))\neq 0$ and ${\det}^r_{G}(C(J))\neq 0$, then 
 \[\tautwo(C_*)\,\,=\,\, \frac{{\det}^r_{G}(B(J,L))}{{\det}^r_{G}(C(J))\cdot {\det}^r_{G}(A(L))}. \]
\end{lemma}

\section{Turaev's algorithm}\label{sec:turaev}
Let $N$ be a connected irreducible 3-manifold and let $\Sigma$ be a Thurston norm minimizing surface.
Recall that $\Sigma\times [-1,1]$ denotes a closed tubular neighborhood of $\Sigma$, $N\doublesms \Sigma$ denotes $N\sms \Sigma \times (-1,1)$ and $\Sigma_\pm$  denotes $\Sigma \times \{\pm 1\}$. 
If $N\doublesms \Sigma$ is disconnected, then this can lead to delicate base point issues.
Turaev's algorithm (described in the following proof of Proposition \ref{prop:turaevsalgorithm}) consists in modifying $\Sigma$ into a surface $S$, without changing the value of $\tautwo(N\doublesms \Sigma,\Sigma_-)$ so that all components of $N\doublesms S$ but possibly one are products. In the next chapter, where we prove the main result of this paper, this result will be crucial in avoiding the aforementioned base point issues.

Let us first state a useful fact:

\begin{proposition}\label{prop:everything-pi1-injective}
Let $N$ be an irreducible 3-manifold  with empty or toroidal boundary.
\bnm[font=\normalfont]
\item If $\Sigma$ is Thurston norm minimizing $($as a reminder, our notion of a Thurston norm minimizing surface excludes in particular compressible tori$)$, then every component of $\Sigma$ is $\pi_1$-injective and every component of $N\doublesms \Sigma$ is $\pi_1$-injective in $N$.
\item Every JSJ-torus and every JSJ-component is $\pi_1$-injective in $N$.
\enm
\end{proposition}

\begin{proof}
The fact that every component of a Thurston norm minimizing surface $\Sigma$ is $\pi_1$-injective is a well-known consequence of the loop theorem, see e.g.\ \cite[(C.22)]{AFW15} for details. The JSJ-tori are $\pi_1$-injective by definition. Both of the remaining statements are now an immediate consequence of \cite[Section 5.2, Corollary 1]{Se03}.
\end{proof}

A \textit{weighted surface} $\widehat{S}$ in a  closed 3-manifold $N$ is a collection of pairs $(S_i,w_i)$, $i=1,\ldots,n$, where $S_i$ are disjoint connected surfaces in $N$ and $w_i$ are positive integers. We denote the union $\bigcup S_i$ by $S'$.

Every weighted surface $\widehat{S}$ defines a homology class $\left [\widehat{S}\right ]:=\sum_{i=1}^{n} w_i\cdot [S_i] \in H_2(N)$. By taking $w_i$ parallel copies of $S_i$ we get a properly embedded surface $S$ such that $[S]=\left [\widehat{S}\right ]$.

Conversely, every  surface $S$ in a  closed 3-manifold $N$ can be seen as a weighted surface by giving every component the weight $1$. 

The first observation regarding $L^2$-torsion, is the following lemma.
\begin{lemma}\label{lem:killingproductpieces}
Let $\what{S}$ be a weighted surface in a closed connected irreducible $3$-manifold $N$. We define $S$ and $S'$  as above. Then $S'$ is a Thurston norm minimizing surface if and only if $S$ is a Thurston norm minimizing surface. In this case we have
\[ \tautwo(N\doublesms S,S_-)\,\,=\,\, \tautwo(N\doublesms S', S'_-) \]
\end{lemma}

\begin{proof}
The first assertion follows from \cite[Corollary~2]{Th86} and \cite[Theorem~5.5]{Ga83}. 
The second one follows from Lemma~\ref{lem:tautwo-basics} (1), since $S$ is obtained by taking parallel copies of some of the components of $S'$.
\end{proof}

The rest of this section is devoted to the proof of the following proposition, where we adapted the proof of \cite[Lemma 1.2]{Tu02} to our setting.

\begin{proposition}\label{prop:turaevsalgorithm} 
	If $\Sigma$ is a Thurston norm minimizing surface in a closed connected irreducible 3-manifold $N$, then there is a weighted surface $\widehat{S}$ in $N$ such that $($with $S$ and $S'$ defined as above$)$:
	\begin{enumerate}[font=\normalfont]
		\item $\left [\widehat{S}\right ]=[\Sigma] \in H_2(N)$,
		\item $S$ is Thurston norm minimizing,
		\item $N\doublesms S'$ is connected, 
		\item $\tautwo(N\doublesms S, S_-)=\tautwo(N\doublesms S', S'_-)=\tautwo(N\doublesms\Sigma, \Sigma_-)$.
	\end{enumerate}
\end{proposition}

In the proof of Proposition~\ref{prop:turaevsalgorithm} we will need the following lemma.
\begin{lemma}\label{lem:l2torsionunchanged}
	Let $N$ be a closed irreducible 3-manifold. Furthermore let $S$ and $T$ be  two Thurston norm minimizing surfaces in $N$ which do not intersect, i.e.\ with $S\cap T=\emptyset$. This means that  $S\cup T$ is an embedded surface. If $N\doublesms (S\cup T)$ decomposes into two $($not necessarily connected$)$ manifolds $N_1$ and $N_2$ such that $\partial N_1= S_-\cup T_+$ and $\partial N_2= T_-\cup S_+$ $($as sketched in Figure \ref{fig:l2torsionunchanged}$)$ then
	\[\tautwo(N\doublesms S, S_-)=\tautwo(N\doublesms T, T_-). \]
\end{lemma}

\begin{figure}[h]
\begin{center}
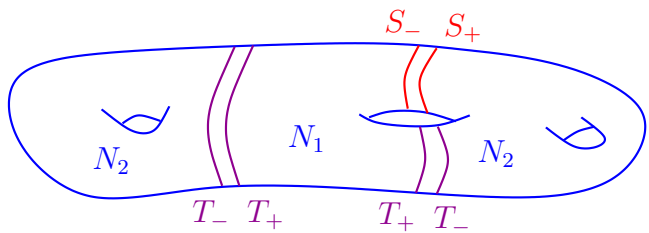
\caption{Illustration of Lemma~\ref{lem:l2torsionunchanged}.}\label{fig:l2torsionunchanged}
\end{center}
\end{figure}

\begin{proof}
	Since $N\doublesms (S\cup T)$ decomposes into two manifolds $N_1$ and $N_2$ such that $\partial N_1=  S_-\cup T_+$ and $\partial N_2=T_-\cup S_+$ we can consider the short exact sequence of chain complexes
	\[
	0\to C_*^{(2)}(N_1,S_-)\to C_*^{(2)}(N\doublesms S,S_-)\to C_*^{(2)}(N\doublesms S, N_1)\to 0 .
	\] 	
	The inclusion defines a natural isomorphism $C_*^{(2)}(N\doublesms S, N_1)= C_*^{(2)}(N_2,T_-)$ of  chain complexes. (The equality in the previous sentence follows from the fact that we work with cellular chain complexes.) Note that the three chain complexes are weakly acyclic and of determinant class by \cite[Theorem~1.1]{Her18} and our hypothesis that $S$ and $T$ are Thurston norm minimizing.
	Then by the multiplicativity  of the $L^2$-torsion \cite[Theorem 3.35(1)]{Lu02} one has
	\[ 
	\tautwo(N_1,S_-)\cdot \tautwo(N_2,T_-)=\tautwo(N\doublesms S,S_-).
	\]
	One could also consider the short exact sequence of chain complexes
	\[
	0\to C_*^{(2)}(N_2,T_-)\to C_*^{(2)}(N\doublesms T,T_-)\to C_*^{(2)}(N\doublesms T, N_2)\to 0.
	\]
	Again we have a natural isomorphism  $C_*^{(2)}(N\doublesms T, N_2)= C_*^{(2)}(N_1,S_-)$ of  chain complexes. Thus, again by multiplicativity, we obtain 
	\[ 
	\tautwo(N_1,S_-)\cdot \tautwo(N_2,T_-)=\tautwo(N\doublesms T,T_-).
	\]
	Combining the above two equalities we obtain that $\tautwo(N\doublesms T,T_-)=\tautwo(N\doublesms S,S_-)$.
\end{proof}

\begin{figure}[!h]
\begin{center}
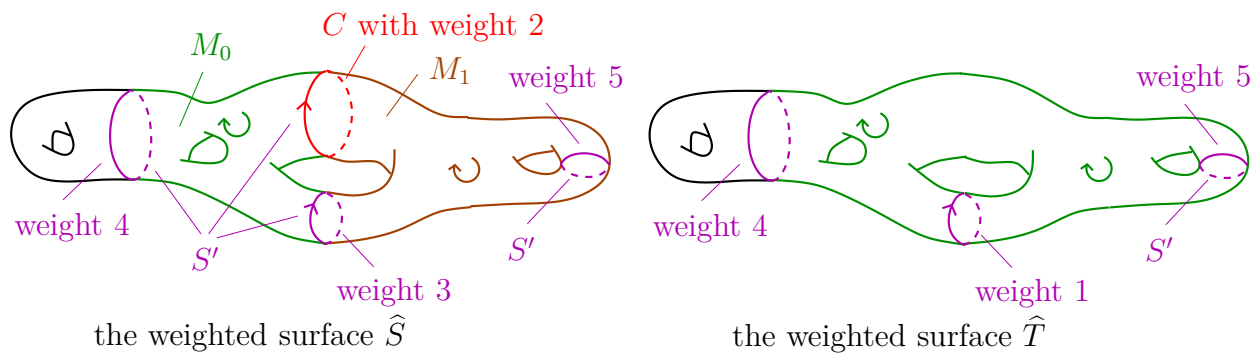
\caption{Illustration of the proof of Proposition~\ref{prop:turaevsalgorithm}.}
\label{fig:turaev-algorithm}
\end{center}
\end{figure}

\begin{proof}[Proof of Proposition~\ref{prop:turaevsalgorithm}]
    Throughout the proof, given any weighted surface $\what{S}$ in $N$ we write $c\left (\what{S}\right ):=\# \pi_0(N\doublesms S')$. 
    Now let $\Sigma$ be a Thurston norm minimizing surface.
    We need to show  that there is a weighted surface $\what{S}$ with Properties (1) to (4).
    First we take $\what{S}$ to be the weighted surface obtained from $\Sigma$ by assign to each component of $\Sigma$ the weight $1$. This weighted surface satisfies Properties  (1), (2) and (4). 
    Clearly we only need to prove that given any weighted surface $\what{S}$ with 
    Properties  (1), (2) and (4) and $c\left (\what{S}\right )>1$ there exists another weighted surface $\what{T}$ with $c\left (\what{T}\right )<c\left (\what{S}\right )$ that still satisfies Properties (1), (2) and (4). 
    
	Let $\what{S}=\left\{(S_i, w_i)\right\}_{i\in I}$ be a weighted surface with Properties (1), (2) and (4) and with $c\left (\what{S}\right )>1$. Since $N$ is connected and $c\left (\what{S}\right )>1$ there exists a component $C\subset S'$ such that $C_+$ and $C_-$ lie in different components of $N\doublesms S'$. Let $C$ be a component with minimal weight among all such components. We denote by $M_0$ and $M_1$ the components of $N\doublesms S'$ containing $C_+$ and $C_-$ respectively.
	
	Note that the boundary $\partial M_1$ comes with a decomposition into two oriented surfaces $R_+:= S_+'\cap M_1$ and $R_-:=S_-'\cap M_1$.  So as homology classes one gets the equality 
	\begin{align}\label{eq:homologyequation}
	[R_+]\,=\,[R_-]\,\in \, H_2(N).
	\end{align}
	
	Moreover, from the assumption that $S$ is Thurston norm minimizing we obtain  that $\chi_-(R_+)=\chi_-(R_-)$. (This can be seen as follows:  the surfaces  $R_+$ and $R_-$ are homologous, and if $\chi_-(R_-)<\chi_-(R_+)$ then we could replace $R_+$ by a parallel copy of  $R_-$ to obtain a surface of lower complexity.)
	Let $w$ be the weight of $C$. We define a new weighted surface in two steps. First as an intermediate step we consider the weighted surface $\left\{(S_i, \tilde{w_i} )\right\}_{i\in I}$, where we have the same underlying surface but the weights change by  
	\[
	\tilde{w_i}\,\:=\,\,\begin{cases}
	w_i+ w& \mbox{ if } {S_i}_+ \subset R_+\\
	w_i &\mbox{ else.}
	\end{cases}
	\]
	Next we define the weighted surface $\widehat{T}:=\left\{(S_i, w'_i)\right\}_{i\in I}$, where the weights are given by
	\[
	w'_i\,\:=\,\,\begin{cases}
	\tilde{w_i}- w& \mbox{ if } {S_i}_- \subset R_-\\
	\tilde{w_i} &\mbox{ else}
	\end{cases}
	\]
	and we add the convention that if $w'_i=0$ then $S_i:=\emptyset$.
	It may happen that $S_{i-}\subset R_-$ and $S_{i+}\subset R_+$. But in this case we have by our definition $w_i=w'_i$. From Equation~\eqref{eq:homologyequation} we obtain  that $\left [\widehat{T}\right ]=\left [\widehat{S}\right ]$. Moreover, $c\left (\what{T}\right )\leqslant c\left (\what{S}\right )-1 < c\left (\what{S}\right )$ since $M_0$ and $ M_1$ lie in the same  component of  $N\doublesms T'$ (see Figure~\ref{fig:turaev-algorithm}) and $T'$ is by construction a subsurface of $S'$.
	
	If we push $T$ slightly in the $-$-direction, then we can make $T$ and $S$ disjoint. Then $S$ and $T$ satisfy the condition in Lemma ~\ref{lem:l2torsionunchanged} (see Figure~\ref{fig:turaev-algorithm} for an illustration). Therefore we have $\tautwo(N\doublesms T, T_-)=\tautwo(N\doublesms S, S_-)$. Hence the properties (1),(2) and (4) do not change under the induction step, which proves the proposition.
\end{proof}

\section{The proof of the main theorem~\ref{main-theorem}} \label{sec:proof}

\subsection{The first inequality in the main theorem~\ref{main-theorem}}

We start out with the following definition.

\begin{definition}
Let $N$ be a connected irreducible $3$-manifold with empty or toroidal boundary and let $\phi \in H^1(N;\Z)$.
We define
\[ A(N,\phi)\,\,:=\,\, \smprod{{\ba{c}{\tiny{\mbox{$M\in \JJ(N)$}}}\\ \tiny{\mbox{with $\phi|_M=0$}}\ea}}{} e^{\vol(M)/6\pi}.\]
\end{definition}

The following lemma says that the first two terms in Theorem \ref{main-theorem} are multiplicative with respect to the JSJ-decomposition. 

\begin{lemma}\label{lem:a-c-multiplicative}
Let $N$ be a connected irreducible $3$-manifold with empty or toroidal boundary and let $\phi \in H^1(N;\Z)$.
Then
\[ A(N,\phi)\,\,=\,\, \smprod{M\in \JJ(N)}{} A(M,\phi|_{M})\quad \mbox{ and }\quad
 C(N,\phi)\,\,=\,\, \smprod{M\in \JJ(N)}{} C(M,\phi|_{M}).\]
\end{lemma}

\begin{proof}
The first equality is an immediate consequence of the definitions.
The second equality follows from Theorem~\ref{thm:previous} (2) and (3).
\end{proof}

Next we consider Seifert fibered spaces and hyperbolic spaces separately. The following lemma implies that the first inequality is an equality for a Seifert fibered space.

\begin{lemma}\label{lem:a-c-for-sfs}
Let $N$ be a Seifert fibered $3$-manifold and let $\phi \in H^1(N;\Z)$.
Then $A(N,\phi)=1$ and $C(N,\phi)=1$.
\end{lemma}

\begin{proof}
The first equality holds by definition, the second equality is proved in 
Theorem \ref{thm:previous2}(1).
\end{proof}

The following proposition is precisely the first inequality in Theorem~\ref{main-theorem} for hyperbolic manifolds. 

\begin{proposition}\label{prop:hyp_first}
Let $N$ be a hyperbolic 3-manifold and let $\phi \in H^1(N;\Z)$.
Then we have $A(N,\phi)\leqslant C(N,\phi)$. 
\end{proposition}

\begin{proof}
If $\phi=0$, then it follows from Lemma~\ref{lem:zero-phi} (2) and Theorem~\ref{thm:previous} (1) that $$A(N,\phi)=e^{\vol(N)/6\pi} = C(N,\phi).$$
Now suppose that $\phi\ne 0$. Then by definition we have $A(N,\phi)=1$. Furthermore by Theorem~\ref{thm:previous2} (4a) we have  $C(N,\phi)\geqslant 1$. 
\end{proof}

\subsection{The second inequality in the  main theorem~\ref{main-theorem}}\label{sub:second}
We now prove the second inequality appearing in Theorem~\ref{main-theorem}.
Put differently, we prove the following proposition.

\begin{proposition}\label{prop:2nd-inequality}
Let $N$ be a connected irreducible $3$-manifold with empty or toroidal boundary.
Furthermore let $\phi \in H^1(N;\Z)$ and $\Sigma$ be a Thurston norm minimizing surface dual to $\phi$. Then
\[
C(N,\phi)
\,\, \leqslant\,\,
\tau^{(2)}(N\doublesms \Sigma, \Sigma_-).\]
\end{proposition}

As the reader will notice, this proposition is technically by far the most involved piece of our paper.

\begin{proof}
Let $N$ be a connected irreducible $3$-manifold with empty or toroidal 
boundary and  $\phi \in H^1(N;\Z)$. Furthermore we let $\Sigma$ be a Thurston norm minimizing surface dual to $\phi$.

If $\phi=0$, then $\Sigma$ is empty and it follows from Lemma~\ref{lem:zero-phi} (2) and the definitions that 
$C(N,0) = \tautwo(N) = \tau^{(2)}(N\doublesms \Sigma, \Sigma_-)$.

In the rest of the proof we assume $\phi$ to be non-zero.
We pick a tubular neighborhood $\Sigma \times [-1,1]$ and we write $M=N\sms \Sigma\times (-1,1)$ and $\Sigma_{\pm}:=\Sigma\times \{\pm 1\}$. 
We denote by $p\colon \wti{N}\to N$ the universal covering. For any subset $X\subset N$ we write $\wti{X}:=p^{-1}(X)$.

We split the proof in three parts. The first one deals with the case when $N$ is closed and $M$ connected, and contains the core of the proof. In the second part we prove the desired inequality for any surface in a closed manifold $N$. In the final part we extend the result to $N$ with toroidal boundary by a doubling argument. Note that the proof of Proposition \ref{prop:2nd-inequality} for knot exteriors does not use the second case.

\begin{step1} $N$ is closed and $M=N \doublesms \Sigma$ is connected.
\end{step1} 

We will first assume that $N$ is closed and that $M$ is \emph{connected}. 
We start out with the following claim.

\begin{claim}
We can find a CW-structure for $N$ with the following properties:
\bnm
\item $M=N\sms \Sigma \times (-1,1)$ and $\Sigma\times [-1,1]$ are subcomplexes,
\item the CW-structure on $\Sigma\times [-1,1]$ is a product structure,
\item $M$ has precisely one 3-cell $\beta$,
\item there is exactly one 0-cell $q$ in the interior $M\sms \Sigma\times \{\pm 1\}$,
\item $\Sigma$ has only one 0-cell 
$p_i$
 in each component $\Sigma_i$,
\item for each $i$ there exist 1-cells  $\nu_i^{\pm}$  going from $q$ to $p_i^\pm = p_i \hspace*{-0.1cm}\times\hspace*{-0.1cm}\{\pm 1\}$ lying entirely in $M$.
\enm
\end{claim}

We sketch the proof of the claim.
We pick a triangulation for $M=N\sms \Sigma \times (-1,1)$. Since all triangulations on surfaces are equivalent after isotopies and subdivisions we can find a triangulation for $M$ such that the triangulations on the two copies $\Sigma\times \{\pm 1\}$ in $M$ agree. We use this triangulation to view $M$ as a CW-complex and we also equip $\Sigma$ as a CW-complex coming from the triangulation. Next we modify the CW-structure to also obtain properties (3), (4), (5) and (6). We do so following an argument of McMullen, see \cite[Proof of Theorem 5.1]{McM02}:
\bnm
\item[(a)] First fuse all the $3$-cells of $M$ along a dual maximal tree to achieve (3).
\item[(b)] We pick a maximal tree in the 1-skeleton on $M$ with the following properties:
\bnm
\item[(i)] the tree  connects all vertices in $M\sms \Sigma_{\pm}$,
\item[(ii)] the tree lies in $M\sms \Sigma_{\pm }$.
\enm
We collapse this tree to a single point $q$. Since any embedded tree in a 3-manifold has a neighborhood that is a ball we see that the collapsed space is again homeomorphic to $M$. But now we have a CW-structure that also satisfies (4) and (6). 
\item[(c)] Finally for each component $\Sigma_i$ of $\Sigma$ we pick a maximal tree $T_i$ in the 1-skeleton of $\Sigma_i$ that connects all vertices. We collapse $T_i\times [-1,1]$. Once again the quotient space is homeomorphic to $M$ and this time we have a CW-structure that has all the desired properties.
\enm
This concludes the proof of the claim.

Next we choose names for the cells of $\Sigma = \Sigma_1 \sqcup \ldots \sqcup \Sigma_l$. More precisely, we denote by $p_i$ the 0-cell of $\Sigma_i$, we write $\mathcal{P}=\{p_i\}$, we denote by $\mathcal{E}=\{e_i\}$ the set of 1-cells and we denote by $\mathcal{F}=\{f_i\}$ the set of 2-cells of $\Sigma$. For clarity, we pick an order on $\mathcal{E}$ (resp. $\mathcal{F}$) so that the cells of $\Sigma_1$ come first, then those of $\Sigma_2$, etc.
We write $I=[-1,1]$. We equip $\Sigma\times I$ with the product CW-structure with cells $p_i^\pm,e_i^\pm,f_i^\pm$ on $\Sigma\times \{\pm 1\}$ and the product cells $p_i\times I,e_i\times I, f_i\times I$ where $i$ runs over the obvious index sets.

The CW-structure on $M$ has $2l+1$ 0-cells, namely $q$ and the $p_i^{\pm}$. We have $1$-cells $\nu_i^\pm$. Let  $\mathcal{M}=\{\mu_i\}$ (resp.\ $\mathcal{S}=\{\sigma_i\}$) be the set of the other 1-cells (resp.\ the set of 2-cells) in the interior of $M$. As a base point we use $q$ and abbreviate $\Lambda:=\Z[\pi_1(N,q)]$.
We denote by $\gamma_i$ the element in $\pi_1(N,q)$ induced by the path given by concatenating $\nu_i^-$, $p_i\times I$ and $(\nu_i^+)^{-1}$ in this order. Note that with our orientation conventions we have $\phi(\gamma_i)=1$ for all $i=1, \ldots, l$.
In the following we use the above cell decomposition and we pick $q$ as our base point.  We connect each cell of $M$ with a path in $M$ to the base point. 
Furthermore we pick paths in $\Sigma_i\times [-1,1)$ together with $\nu_i^{-}$ 
to connect the cells in $\Sigma_i\times [-1,1)$ to the base point $q$. Finally we pick paths in $\Sigma_i\times \{1\}$ and $\nu_i^+$ to connect the cells in $\Sigma_i\times \{1\}$ to the base point $q$ (see Figure \ref{fig:base-points-paths}).  These choices of  paths correspond to choices of lifts of the cells to the universal covering of $N$. We use these lifts as the basis of the cellular chain complex $C_*\left (\wti{N}\right )$, viewed as a free left $\Z[\pi_1(N,q)]$-module. 
\begin{align*}
C_3\left (\wti{N}\right )&=
\Lambda\cdot \widetilde{\beta} 
\oplus 
\left(\,\tmoplus{i=1}{|\mathcal{F}|}\Lambda\cdot \widetilde{f_i\times I} \right)
,\\
C_2\left (\wti{N}\right )&=
\left (\,\tmoplus{i=1}{|\mathcal{S}|}\Lambda\cdot \widetilde{\sigma_i} \right )
\oplus
\left (\,\tmoplus{i=1}{|\mathcal{F}|}\Lambda\cdot \widetilde{f^+_i} \right )
\oplus
\left (\,\tmoplus{i=1}{|\mathcal{F}|}\Lambda\cdot \widetilde{f^-_i} \right )
\oplus 
\left (\,\tmoplus{i=1}{|\mathcal{E}|}\Lambda\cdot \widetilde{e_i\times I} \right )
, \\ 
C_1\left (\wti{N}\right )&=
\left (\,\tmoplus{i=1}{l}\Lambda\cdot \widetilde{p_i \times I} \right )
\oplus
\left (\,\tmoplus{i=1}{l}\Lambda\cdot \widetilde{\nu_i^+} \right )
\oplus
\left (\,\tmoplus{i=1}{l}\Lambda\cdot \widetilde{\nu_i^-} \right ) \\
& \hspace*{0.5cm}
\oplus
\left (\,\tmoplus{i=1}{|\mathcal{M}|}\Lambda\cdot \widetilde{\mu_i} \right )
\oplus
\left (\,\tmoplus{i=1}{|\mathcal{E}|}\Lambda\cdot \widetilde{e^+_i} \right )
\oplus 
\left (\,\tmoplus{i=1}{|\mathcal{E}|}\Lambda\cdot \widetilde{e^-_i} \right )
, \\ 
C_0\left (\wti{N}\right )&=
\left (\,\tmoplus{i=1}{l}\Lambda\cdot \widetilde{p_i^+} \right )
\oplus
\left (\,\tmoplus{i=1}{l}\Lambda\cdot \widetilde{p_i^-} \right )
\oplus
\Lambda\cdot \widetilde{q} ,
\end{align*}
\begin{figure}[h]
\begin{center}
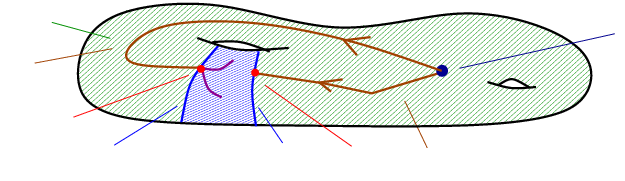\caption{The cell decomposition of $M$}
\label{fig:base-points-paths}
\end{center}
\end{figure}
Recall that the boundary map in the cellular chain complex is given in the usual way, except that the coefficient in $\Z[\pi_1(N,q)]$ corresponds to the concatenation of the preferred path in the original cell and the concatenation of the preferred path of the cell in the boundary.
It is now fairly straightforward to see that the boundary maps are given as follows (the notation  used in the matrices will be explained just below the matrices):

\renewcommand{\kbldelim}{(}
\renewcommand{\kbrdelim}{)}

\begin{align*}
\partial_3^N:=
\kbordermatrix{\mbox{}& \widetilde{\mathcal{S}} & \widetilde{\mathcal{F}^+} & \widetilde{\mathcal{F}^-} & \widetilde{\mathcal{E}\times I} \\
\widetilde{\beta} & A & -1 & 1 & 0\\
\widetilde{\mathcal{F}\times I} & 0 & d(\gamma_i) & -\id & \partial_{\Sigma}^2\times I},
 \\
\partial_2^N:=
\kbordermatrix{\mbox{}& \widetilde{p_i \times I} & \widetilde{\nu_i^+} & \widetilde{\nu_i^-} & \widetilde{\mathcal{M}} & \widetilde{\mathcal{E}^+} & \widetilde{\mathcal{E}^-} \\
\widetilde{\mathcal{S}} & 0 & B^+ & B^- & C & D^+ & D^-\\
\widetilde{\mathcal{F}^+} & 0 & 0 & 0 & 0 & \partial_{\Sigma,+}^2 & 0\\
\widetilde{\mathcal{F}^-} & 0 & 0 & 0 & 0 & 0 & \partial_{\Sigma,-}^2\\
\widetilde{\mathcal{E}\times I} & \partial_{\Sigma}^1\times I & 0 & 0 & 0 & -d(\gamma_i) & \id },
\\
\partial_1^N:=
\kbordermatrix{\mbox{}&\widetilde{p_i^+} & \widetilde{p_i^-} & \widetilde{q} \\
\widetilde{p_i \times I} & d(\gamma_i) & -\id & 0\\
\widetilde{\nu_i^+} & \id & 0 & -1\\
\widetilde{\nu_i^-} & 0 & \id & -1\\
\widetilde{\mathcal{M}} & E^+ & E^- & F\\
\widetilde{\mathcal{E}^+} & \partial_{\Sigma,+}^1 & 0 & 0\\
\widetilde{\mathcal{E}^-} & 0 & \partial_{\Sigma,-}^1 & 0}
.
\end{align*}
It is perhaps worth recalling that we think of elements in $\Lambda^k$ as row vectors and that we multiply by matrices from the right. 
Also let us clarify some notations. For each $i=1, \ldots, l$, the pointed topological spaces $(\Sigma_i^+,p_i^+), (\Sigma_i^-,p_i^-), (\Sigma_i \times I,p_i^-)$ are naturally homotopy equivalent.
We use these natural homotopy equivalences to identify their fundamental groups and we denote this common fundamental group by $\pi_1(\Sigma_i)$. Now for each $j=0,1,2$ the symbols $\partial_{\Sigma+}^j$ and $\partial_{\Sigma,-}^j$ denote the boundary operators in the cellular chain complexes of the universal covers and are written as the same block diagonal matrix with blocks over $\Z[\pi_1(\Sigma_1)], \ldots, \Z[\pi_1(\Sigma_l)]$. Since $\Sigma_i$ is Thurston-norm minimizing
we know from Proposition~\ref{prop:everything-pi1-injective} that the  inclusion of each component of $\Sigma_i$ into $M$ and $N$ is $\pi_1$-injective, and thus by a slight abuse of notation we denote again by $\partial_{\Sigma+}^j$ and $\partial_{\Sigma,-}^j$ their inductions over $\Lambda$ through $\pi_1(\Sigma_i) \hookrightarrow \pi_1(N,p_+)$. Similarly, $A,B^+,B^-,C,D^+,D^-,E^+, E^-,F$ are matrices over $\Z[\pi_1(M,q)] \subset \Lambda$ that represent pieces of boundary operators of $C_*\left (\wti{M}\right )$ (the inclusion $\pi_1(M,q) \hookrightarrow \pi_1(N,q)$ comes from Proposition \ref{prop:everything-pi1-injective}). Moreover, a $0,1$ or $-1$ in a box means all coefficients of the box are equal to this number. Finally $d(\gamma_i)$ means the diagonal matrix with 
an appropriate number of $\gamma_i$-entries on the diagonal, where the number of $\gamma_i$-entries is determined by the number of cells of $\mathcal{P}$, $\mathcal{E}$ or $\mathcal{F}$ that lie in $\Sigma_i$.

By construction we have $\phi(\gamma_i)=1$ for all $i$ and $\pi_1(M,q)\subset \ker \phi$. Hence the boundary matrices $\partial^{N,(2)}_j(t)$ ($j=1,2,3$) for the $L^2$-chain complex $C_*^{(2)}(N,\phi,t)$ are given by right multiplications by the matrices $\partial^N_j$ where the $d(\gamma_i)$ are simply changed to $d(t \cdot \gamma_i)$.

Using the same CW-structure for the pair $(M,\Sigma_-)$ and denoting $\Lambda' := \Z[\pi_1(M,q)]$ we obtain for the cellular chain complex 
$C_*\left (\wti{M},\wti{\Sigma_-}\right )$
 the decomposition 
\begin{align*}
C_3\left (\wti{M},\wti{\Sigma_-}\right )&=
\Lambda'\cdot \widetilde{\beta} 
,\\
C_2\left (\wti{M},\wti{\Sigma_-}\right )&=
\left (\,\tmoplus{i=1}{|\mathcal{S}|}\Lambda'\cdot \widetilde{\sigma_i} \right )
\oplus
\left (\,\tmoplus{i=1}{|\mathcal{F}|}\Lambda'\cdot \widetilde{f^+_i} \right )
, \\ 
C_1\left (\wti{M},\wti{\Sigma_-}\right )&=
\left (\,\tmoplus{i=1}{l}\Lambda'\cdot \widetilde{\nu_i^+} \right )
\oplus
\left (\,\tmoplus{i=1}{l}\Lambda'\cdot \widetilde{\nu_i^-} \right )
\oplus
\left (\,\tmoplus{i=1}{|\mathcal{M}|}\Lambda'\cdot \widetilde{\mu_i} \right )
\oplus
\left (\,\tmoplus{i=1}{|\mathcal{E}|}\Lambda'\cdot \widetilde{e^+_i} \right )
, \\ 
C_0\left (\wti{M},\wti{\Sigma_-}\right )&=
\left (\,\tmoplus{i=1}{l}\Lambda'\cdot \widetilde{p_i^+} \right )
\oplus
\Lambda'\cdot \widetilde{q},
\end{align*}
and with this direct sum decomposition the boundary matrices are given by $$
\partial_3^M:=
\kbordermatrix{\mbox{}& \widetilde{\mathcal{S}} & \widetilde{\mathcal{F}^+} \\
\widetilde{\beta} & A & -1 }
,
 \
\partial_2^M:=
\kbordermatrix{\mbox{}& \widetilde{\nu_i^+} & \widetilde{\nu_i^-} & \widetilde{\mathcal{M}} & \widetilde{\mathcal{E}^+} \\
\widetilde{\mathcal{S}} & B^+ & B^- & C & D^+ \\
\widetilde{\mathcal{F}^+} & 0 & 0 & 0 & \partial_{\Sigma,+}^2 
},\
\partial_1^M:=
\kbordermatrix{\mbox{}&\widetilde{p_i^+} & \widetilde{q}\\
\widetilde{\nu_i^+} & \id & -1 \\
\widetilde{\nu_i^-} & 0 & -1 \\
\widetilde{\mathcal{M}} & E^+ & F \\
\widetilde{\mathcal{E}^+} & \partial_{\Sigma,+}^1 & 0
}.
$$

Similarly as before, the boundary matrices $\partial^{M,(2)}_j$ ($j=1,2,3$) for the $L^2$-chain complex $C_*^{(2)}(M,\Sigma_-)$ are given by right multiplications by the matrices $\partial^M_j$.

We write $G'=\pi_1(M,p_+)$. We apply Lemma~\ref{lem:calculatetorsion} to the chain complex
\[ C_*^{(2)}(M,\Sigma_-)\quad \mbox{ with } J\,=\, \big(\widetilde{f_{|\mathcal{F}|}^+}\big)
\mbox{ and } L\,=\,\big(\widetilde{\nu_1^+},\ldots,\widetilde{\nu_l^+},\widetilde{\nu_1^-} \big).\]
Together with  Proposition \ref{prop:op_det} we obtain that
\begin{align*}
\tautwo(M,\Sigma_-) = \tautwo\left (C_*^{(2)}(M,\Sigma_-)\right )
&=
\dfrac{{\det}_{G'}^r\begin{pmatrix}
B_{2,\ldots,l}^- & C & D^+ \\
0 & 0 & \partial^+
\end{pmatrix} }
{ {\det}_{G'}^r\begin{pmatrix}
-1 \end{pmatrix} {\det}_{G'}^r\begin{pmatrix}
\id & -1 \\
0 & -1
\end{pmatrix} }
={\det}_{G'}^r\begin{pmatrix}
B_{2,\ldots,l}^- & C & D^+ \\
0 & 0 & \partial^+
\end{pmatrix}
\hspace*{-0.1cm},
\end{align*}
where $\partial^+$ is $\partial_{\Sigma,+}^2 $ without its last row and $B_{2,\ldots,l}^-$ is $B^-$ without its first column.

Similarly, we consider the pair $(M,\Sigma_+)$ and we compute
\[
\tautwo(M,\Sigma_+) = {\det}_{G'}^r\begin{pmatrix}
B_{1,\ldots,l-1}^+ & C & D^- \\
0 & 0 & \partial^-
\end{pmatrix}
\]
where $\partial^-$ is $\partial_{\Sigma,-}^2$ without its first row.

We write $G=\pi_1(N,p_+)$.
Let 
\[ f(t)\,:=\,\tautwo\left (C_*^{(2)}(N,\phi,t)\right )\]
denote the particular representative of the $L^2$-Alexander torsion $\tautwo(N,\phi,t)$  computed from the previous choices of cells and lifts.

Next we apply Lemma~\ref{lem:calculatetorsion} to the chain complex
\[ C_*^{(2)}(N,\phi,t)\mbox{ with }J\,=\,\big\{\widetilde{f_{|\mathcal{F}|}^+},\widetilde{\mathcal{F}^-}\big\}\mbox{ and }
L\,=\, \big(\widetilde{p_1\times I},\ldots,\widetilde{p_l\times I},\widetilde{\nu_1^+},\ldots,\widetilde{\nu_l^+},\widetilde{\nu_1^-} \big).\]
Together with Proposition \ref{prop:op_det} we obtain that 
\begin{align*}
f(t)
&=
\dfrac{
{\det}_G^r\begin{pmatrix}
B_{2,\ldots,l}^- & C & D^+ & D^- \\
0 & 0 & \partial^+ & 0 \\
0 & 0 & - d(t \cdot \gamma_i) & \id
\end{pmatrix}
}{
{\det}_G^r\begin{pmatrix}
-1 & 1 & \ldots & 1 \\
0 & \\
\vdots & \\ 
0 & \ & -\id \\
t\cdot \gamma_l
\end{pmatrix}
{\det}_G^r\begin{pmatrix}
d(t \cdot \gamma_i) & -\id & 0 \\
\id & 0 & -1\\
0 & 1 \ 0 \ldots 0 & -1
\end{pmatrix}
}
\\
&=\lmfrac{1}{\max\left\{1,t\right\}^2} \cdot{\det}_G^r\begin{pmatrix}
B_{2,\ldots,l}^- & C & D^+ & D^- \\
0 & 0 & \partial^+ & 0 \\
0 & 0 & - d(t\cdot \gamma_i) & \id
\end{pmatrix}
= \dfrac{g(t)}{\max\{1,t\}^2},
\end{align*}
where $g(t) := {\det}_G^r\begin{pmatrix}
B_{2,\ldots,l}^- & C & D^+ & D^- \\
0 & 0 & \partial^+ & 0 \\
0 & 0 & - d(t\cdot \gamma_i) & \id
\end{pmatrix}$.

We know from Theorem~\ref{thm:previous2} (4) that there exists a $k \in \R$ such that $f(t) \sim_{t \to 0^+} C(N,\phi) t^k$ and $f(t) \sim_{t \to \infty} C(N,\phi) \cdot t^{k+x_N(\phi)}$. In the remainder of Step 1 we will prove that $k=0$ and that $C(N,\phi) \leqslant \tautwo(M,\Sigma_-)$. More precisely, we will prove successively the following statements:
\bnm
\item[(a)] the inequality $k \geqslant 0$, 
\item[(b)] we prove that $k=0$ implies $C\left(N,{\phi}\right)\leqslant  \tautwo(M,\Sigma_-)$,
\item[(c)] the inequality $k \leqslant 0$.
\enm

First we prove (a) and (b). To do this, we take $t=1/n$ in $f(t) \sim_{t \to 0^+} C(N,\phi)t^k$, and we obtain
$$\dfrac{C(N,\phi)}{n^k} \sim_{n \to \infty} f\left (\tmfrac{1}{n}\right ) = g\left (\tmfrac{1}{n}\right ) = 
{\det}_G^r\begin{pmatrix}
B_{2,\ldots,l}^- & C & D^+ & D^- \\
0 & 0 & \partial^+ & 0 \\
0 & 0 & -d\left (\frac{1}{n} \cdot \gamma_i\right ) & \id
\end{pmatrix}.$$
Since from Lemma~\ref{lem:uppersemicont} and Proposition \ref{prop:op_det} we have that
$$\limsup_{n \to \infty} g\left (\tmfrac{1}{n}\right ) \leqslant 
{\det}_G^r\begin{pmatrix}
B_{2,\ldots,l}^- & C & D^+ & D^- \\
0 & 0 & \partial^+ & 0 \\
0 & 0 & 0 & \id
\end{pmatrix}
 = 
{\det}_{G'}^r\begin{pmatrix}
B_{2,\ldots,l}^- & C & D^+ \\
0 & 0 & \partial^+ 
\end{pmatrix} = \tautwo(M,\Sigma_-), $$
then the previous asymptotics imply that 
$\limsup_{n \to \infty} \dfrac{C(N,\phi)}{n^k} \leqslant \tautwo(M,\Sigma_-)$. Since $C(N,\phi)>0$ and $\tautwo(M,\Sigma_-)<+\infty$, it follows that $k\geqslant 0$. 

Furthermore the same inequality $\limsup_{n \to \infty} \dfrac{C(N,\phi)}{n^k} \leqslant \tautwo(M,\Sigma_-)$ shows that if $k=0$, then we get the desired inequality $C(N,\phi) \leqslant \tautwo(M,\Sigma_-)$. 

Finally we prove (c), which concludes the proof of Step 1.
We apply Lemma~\ref{lem:calculatetorsion}  to 
\[ C_*^{(2)}(N,\phi,t)\mbox{ but now with }J=
\big(\widetilde{\mathcal{F}^+},\widetilde{f_{1}^-}\big)\mbox{ and }L=\big(\widetilde{p_1\times I},\ldots,\widetilde{p_l\times I},\widetilde{\nu_l^+},\widetilde{\nu_1^-},\ldots,\widetilde{\nu_l^-} \big).\]
Together with  Proposition \ref{prop:op_det} we obtain that
\begin{align*}
f(t)
&=
\dfrac
{
{\det}_G^r\begin{pmatrix}
B_{1,\ldots,l-1}^+ & C & D^+ & D^- \\
0 & 0 & 0 & \partial^- \\
0 & 0 & -d(t \cdot \gamma_i) & \id
\end{pmatrix}
}
{
{\det}_G^r\begin{pmatrix}
d(t \cdot \gamma_i) & -\id & 0 \\
0 \ldots 0 \ 1 & 0 & -1 \\
0 & \id & -1
\end{pmatrix}
\
{\det}_G^r\begin{pmatrix}
-1  \ldots  -1 & 1 \\
 & -1 \\
 & 0 \\
d( t\cdot \gamma_i)  & \vdots \\ 
 & 0 \\
\end{pmatrix}
}
 \\
 &=
\dfrac{1}{t^{|\mathcal{F}|-1}\max\{1,t\}}
\cdot 
{\det}_G^r\begin{pmatrix}
B_{1,\ldots,l-1}^+ & C & D^+ & D^- \\
0 & 0 & 0 & \partial^- \\
0 & 0 & -t \cdot d(\gamma_i) & \id
\end{pmatrix}
\cdot 
\dfrac{1}{t^{l-1}\max\{1,t\}}
 \\
 &=
\dfrac{1}{t^{l+|\mathcal{F}|-2}\max\{1,t\}^2}
\cdot t^{|\mathcal{E}|} \cdot 
{\det}_G^r\begin{pmatrix}
B_{1,\ldots,l-1}^+ & C & D^+ & D^- \\
0 & 0 & 0 & \partial^- \\
0 & 0 & - d(\gamma_i) & t^{-1}\id
\end{pmatrix}\\
 &=
\dfrac{t^{-\chi(\Sigma)}}{t^{-2}\max\{1,t\}^2}
\cdot 
{\det}_G^r\begin{pmatrix}
B_{1,\ldots,l-1}^+ & C & D^+ & D^- \\
0 & 0 & 0 & \partial^- \\
0 & 0 & - d(\gamma_i) & t^{-1}\id
\end{pmatrix}
=
\dfrac{t^{x_N(\phi)}\cdot h(t)}{t^{-2}\max\{1,t\}^2}
,
\end{align*}
where $h(t) := {\det}_G^r\begin{pmatrix}
B_{1,\ldots,l-1}^+ & C & D^+ & D^- \\
0 & 0 & 0 & \partial^- \\
0 & 0 & - d(\gamma_i) & t^{-1}\id
\end{pmatrix}$. Here note that the final two equalities come from the fact that
$ -|\mathcal{F}|+|\mathcal{E}|-l= - \chi(\Sigma) = x_N(\phi)
$.

Since $C(N,\phi)t^{k+x_N(\phi)} \sim_{t \to \infty} f(t)
\sim_{t \to \infty} t^{x_N(\phi)} h(t)$, then $C(N,\phi)t^{k} \sim_{t \to \infty} h(t)$.

As before, by taking $t=n$ and using Lemma~\ref{lem:uppersemicont} and Proposition \ref{prop:op_det} we conclude that
$$\limsup_{n \to \infty} h(n) \leqslant 
{\det}_G^r
\hspace*{-0.1cm}
\begin{pmatrix}
B_{1,\ldots,l-1}^+ & C & D^+ & D^- \\
0 & 0 & 0 & \partial^- \\
0 & 0 & - d(\gamma_i) & 0
\end{pmatrix}
\hspace*{-0.1cm}
 = 
{\det}_{G'}^r
\hspace*{-0.1cm}
\begin{pmatrix}
B_{1,\ldots,l-1}^+ & C  & D^- \\
0 & 0  & \partial^- 
\end{pmatrix} 
\hspace*{-0.1cm}
= \tautwo(M,\Sigma_+), $$
 thus  $  \limsup_{n \to \infty} C(N,\phi)n^{k} = \limsup_{n \to \infty} h(n) \leqslant \tautwo(M,\Sigma_+)$. Since $C(N,\phi)>0$ and $\tautwo(M,\Sigma_+)<+\infty$, it follows that $k\leqslant 0$, which concludes the proof of Step 1.

\begin{step2}
Surfaces in a closed $N$.
\end{step2}

We let $N$ be a closed 3-manifold and as before suppose we are given $\phi \in H^1(N;\Z)$ and  a Thurston norm minimizing surface $\Sigma$  dual to $\phi$. Now we no longer assume that 
$N\doublesms\Sigma$ is connected. By Proposition \ref{prop:turaevsalgorithm}
there exists  a weighted surface $\what{S}= \{(S_i, w_i)\}_{i=1, \ldots, l}$ in $N$  with
\begin{enumerate}[font=\normalfont]
		\item the class  $\ttmsum{i=1}{l} w_i\cdot [S_i]\in H_2(N)$ is dual to $\phi$,
		\item we have $-\ttmsum{i=1}{l} w_i\cdot \chi(S_i)=x_N(\phi)$,
			\item if we set $S'=\cup S_i$, then $N\doublesms S'$ is connected,
		\item $\tautwo(N\doublesms S', S'_-)=\tautwo(N\doublesms\Sigma, \Sigma_-)$.
	\end{enumerate}
Thus it remains to prove the following claim.

\begin{claim}
We have
$C\left (N,\phi\right )\leqslant \tautwo(N\doublesms S', S'_-).$
\end{claim}

The proof of the claim is basically the same as the proof of Step 1, except that now, perhaps somewhat confusingly,  $S'$ plays the role of the surface $\Sigma$ of Step 1. 

To calculate $\tautwo\left (N,\phi\right )(t)$, we use the same type of  CW-complex structure for $N$ as in Step 1, starting from the decomposition $N=(N\sms S'\times [-1,1])\cup (S'\times [-1,1])$. The difference is  that now we have  $\phi(\gamma_i)=w_i$ for all $i$, and $w_i$ is not necessarily equal to $1$. Then we obtain for the torsion $\tautwo\left (N,\phi\right )(t)$ a representative $f(t)$ of the form
\begin{align*}
f(t)
&=
\dfrac{
{\det}_G^r\begin{pmatrix}
B_{2,\ldots,l}^- & C & D^+ & D^- \\
0 & 0 & \partial^+ & 0 \\
0 & 0 & - d(t^{w_i}\gamma_i) & \id
\end{pmatrix}
}{
{\det}_G^r\begin{pmatrix}
 d(t^{w_i} \gamma_i) & -\id & 0 \\
\id & 0 & -1\\
0 & 1 \ 0 \ldots 0 & -1
\end{pmatrix}
\
{\det}_G^r\begin{pmatrix}
-1 & 1 & \ldots & 1 \\
0 & \\
\vdots & \\ 
0 & \ & -\id \\
t^{w_l}\cdot \gamma_l
\end{pmatrix}
}
\\
&=\lmfrac{1}{\max\left\{1,t\right\}^{w_1+w_l}} \cdot{\det}_G^r\begin{pmatrix}
B_{2,\ldots,l}^- & C & D^+ & D^- \\
0 & 0 & \partial^+ & 0 \\
0 & 0 & -d(t^{w_i}\gamma_i) & \id
\end{pmatrix}
=: \dfrac{g(t)}{\max\{1,t\}^{w_1+w_l}}
.
\end{align*}
Here, analogously to the notation in Step 1 we denote by $d(t^{w_i}\gamma_i)$ the diagonal matrix with an appropriate number of terms $t^{w_i}\gamma_i$ on the diagonal.
Recall that there exists $k \in \R$ such that $f(t) \sim_{t \to 
0^+} C\left (N,\phi\right ) \cdot t^k$ and 
$f(t) \sim_{t \to \infty} C\left (N,\phi\right 
) \cdot t^{k+x_N\left (\phi\right )}$. Like in Step 1 
we will prove
in succession the following three statements:
\bnm
\item[(a)] we have $k \geqslant 0$,
\item[(b)] $k=0$ implies that $ C\left (N,\phi\right )\leqslant \tautwo(N\doublesms S', S'_-) $,
\item[(c)] we have $k \leqslant 0$.
\enm
The proofs of statements (a), (b) and (c)  are very similar to the proofs of the corresponding statements in Step 1. We just need to carry the weights along and we need to verify that nothing goes awry.

First we prove (a). 
With similar reasoning as in Step 1 we obtain 
$$ \limsup_{n \to \infty} \dfrac{C\left (N,\phi\right )}{n^k} =
\limsup_{n \to \infty} g\left (\tmfrac{1}{n}\right )
\leqslant \tautwo(N\doublesms S',S'_-).$$
By Lemma~\ref{lem:killingproductpieces} and Theorem~\ref{thm:reltorexist} we know that $\tautwo(N\doublesms S',S'_-)\ne 0$. Thus we obtain  that $k \geqslant 0$. 
 
Moreover the previous inequality
\[ \limsup_{n \to \infty} \dfrac{C\left (N,\phi\right )}{n^k}  \,\,\leqslant\,\,\tautwo(N\doublesms S', S'_-),\]
notably implies (b).

Finally it remains to prove (c). We will proceed mostly like in Step 1, but we need to introduce some extra notation. Let $\mathcal{P}_i,\mathcal{E}_i$ and  $\mathcal{F}_i$ be the sets of $0$,$1$ and $2$-cells of $S_i$. We write
$$
|\mathcal{P}|_w := \sum_{i=1}^{l} w_i\cdot |\mathcal{P}_i| = \sum_{i=1}^l w_i, \hspace*{1cm}
|\mathcal{E}|_w := \sum_{i=1}^{l} w_i\cdot  |\mathcal{E}_i|, \hspace*{1cm}
|\mathcal{F}|_w := \sum_{i=1}^{l} w_i \cdot |\mathcal{F}_i|.
$$
With this convenient notation one has $x_N\left (\phi \right ) = - \ttmsum{i=1}{l}w_i\cdot \chi(S_i) = -|\mathcal{F}|_w+|\mathcal{E}|_w-|\mathcal{P}|_w$.

Now we compute $f(t)$ the same way as  in Step 1 and obtain:
\begin{align*}
f(t)
&=
\dfrac{
{\det}_G^r\begin{pmatrix}
B_{1,\ldots,l-1}^+ & C & D^+ & D^- \\
0 & 0 & 0 & \partial^- \\
0 & 0 & - d(t^{w_i}\gamma_i) & \id
\end{pmatrix}
}{
{\det}_G^r\begin{pmatrix}
-1  \ldots  -1 & 1 \\
 & -1 \\
 & 0 \\
 d(t^{w_i}\gamma_i)  & \vdots \\ 
 & 0 \\
\end{pmatrix}\hspace{-0.3cm}
\cdot
{\det}_G^r\begin{pmatrix}
d(t^{w_i} \gamma_i) & -\id & 0 \\
0 \ldots 0 \ 1 & 0 & -1 \\
0 & \id & -1
\end{pmatrix}
}
 \\
 &=
\dfrac{1}{t^{|\mathcal{F}|_w-w_1}\max\{1,t\}^{w_1}}
\cdot 
{\det}_G^r\begin{pmatrix}
B_{1,\ldots,l-1}^+ & C & D^+ & D^- \\
0 & 0 & 0 & \partial^- \\
0 & 0 & d(t^{w_i}\gamma_i) & \id
\end{pmatrix}
\cdot 
\dfrac{1}{t^{w_1+\ldots+w_{l-1}}\max\{1,t\}^{w_l}}
 \\
 &=
\dfrac{1}{t^{|\mathcal{P}|_w+|\mathcal{F}|_w-w_1-w_l}\max\{1,t\}^{w_1+w_l}}
\cdot t^{|\mathcal{E}|_w} \cdot 
{\det}_G^r\begin{pmatrix}
B_{1,\ldots,l-1}^+ & C & D^+ & D^- \\
0 & 0 & 0 & \partial^- \\
0 & 0 & - d(\gamma_i) & d(t^{-w_i})
\end{pmatrix}\\
 &=
\dfrac{t^{-\chi(S)}}{t^{-w_1-w_l}\max\{1,t\}^{w_1+w_l}}
\cdot 
{\det}_G^r\begin{pmatrix}
B_{1,\ldots,l-1}^+ & C & D^+ & D^- \\
0 & 0 & 0 & \partial^- \\
0 & 0 & - d(\gamma_i) & d(t^{-w_i})
\end{pmatrix} \\
&=:
\dfrac{t^{x_N(\phi)}\cdot  h(t)}{t^{-w_1-w_l}\max\{1,t\}^{w_1+w_l}}.
\end{align*}
By the same reasoning as in Step 1 we have
$$\limsup_{n \to \infty} C\left (N,\phi\right)  n^k
=\limsup_{n \to \infty} h(n) \leqslant \tautwo(N\doublesms S',S'_+),$$
and therefore $k \leqslant 0$.

\begin{step3}
The 3-manifold $N$ has non-empty toroidal boundary.
\end{step3}

Finally suppose that $N$ is a 3-manifold with non-empty toroidal boundary.
Furthermore let $\phi \in H^1(N;\Z)$ be non-zero
and let $\Sigma$ be a Thurston norm minimizing surface dual to $\phi$. 
Let $N'$ be a copy of $N$ and let $\Sigma'\subset N'$ be the copy of $\Sigma$.
Let $W:=N\cup_{\partial N=\partial N'}N'$ be the double of $N$ and let $S:=\Sigma\cup_{\partial \Sigma=\partial \Sigma'}\Sigma'$ be the double of $\Sigma$.  Furthermore let $r\colon W\to N$ be the obvious retraction onto $N$. 
A short argument shows that $S$ is dual to $r^*\phi$.
Since $\Sigma$ and $\Sigma'$ are Thurston norm minimizing it follows from a standard argument, see 
\cite[Proposition~3.5]{EN85},  that $S$ is also Thurston norm minimizing. 
 Then
\[\ba{rl}
C(N,\phi)^2&=\, C(N,\phi)\cdot C\left (N',r^*(\phi)\right )\,=\,C(W,r^*\phi)
\,\, \leqslant\,\,
\tau^{(2)}(W\doublesms S,S_-)\\
&=\,\,
\tau^{(2)}(N\doublesms \Sigma, \Sigma_-)\cdot 
\tau^{(2)}(N'\doublesms \Sigma', \Sigma'_-)\,=\,
\tau^{(2)}(N\doublesms \Sigma, \Sigma_-)^2.\ea \]
Here the first equality follows from the fact that $r$ restricts to a diffeomorphism $N'\to N$. The second equality follows from Theorem~\ref{thm:previous} (2) and (3) together, i.e. the multiplicativity of $L^2$ torsion under gluings along tori.
The inequality stems from the inequality for closed manifolds that we had proved in the first two steps. The third equality is a consequence of 
Lemma~\ref{lem:tautwo-basics} (5) and the observation that the relative $L^2$-torsion of the intersection of the pairs
$(N\doublesms \Sigma, \Sigma_-)$ and $(N'\doublesms \Sigma', \Sigma'_-)$ is trivial. (This is a consequence of Lemma~\ref{lem:tautwo-basics} (3) and because each component of the intersection is given by tori, annuli or circles.) Finally the last equality follows again from the observation that $r$ restricts to a diffeomorphism $(N'\doublesms \Sigma',\Sigma'_-)\to (N\doublesms \Sigma,\Sigma_-)$. 
\end{proof}

\subsection{Conclusion of the proof}
Now we can finally complete the proof of Theorem~\ref{main-theorem}.

\begin{proof}[Proof of Theorem~\ref{main-theorem}]

The first inequality in the theorem is a consequence of Lemma \ref{lem:a-c-multiplicative}, Lemma \ref{lem:a-c-for-sfs}, and Proposition \ref{prop:hyp_first}.

The second inequality in the theorem is exactly Proposition \ref{prop:2nd-inequality}.
\end{proof}

\section{Ahyperbolic surfaces}\label{sec:ahyperbolic}

In this section we introduce the class of ahyperbolic surfaces. If $\Sigma$ is ahyperbolic then we show that $\tautwo(N\doublesms\Sigma, \Sigma_-)= 1$ and hence all terms in Theorem~\ref{main-theorem} are equal to $1$. Before we give the definition of ahyperbolic surfaces we recall the definition of a sutured manifold.
\begin{definition}
	A \textit{sutured manifold} is a quadruple $(M,R_+,R_-,\gamma)$ where $M$ is a 3-manifold
	with a partition of its boundary $\partial M$ into two subsurfaces $R_+$ and $R_-$ along their common boundary $\gamma$ (the components of $\gamma$ are called the \textit{sutures}). The surface $R_+$ is oriented by the outward-pointing normal and $R_-$ is oriented by the inward pointing one.
	
	A sutured manifold $(M,R_+,R_-,\gamma)$ is called \textit{taut} if the surfaces
	$R_{\pm}$ are both \textit{taut}, i.e.\ $R_{\pm}$ do not contain any disc component and have minimal complexity among all surfaces representing $[R_{\pm}]\in H_2(M,\gamma)$.
\end{definition}

\begin{exemple}
For a Thurston norm minimizing surface $\Sigma$ in a 3-manifold $N$ the manifold $N \doublesms \Sigma$ is a taut sutured manifold with the sutured manifold structure
\[ \left (N \doublesms \Sigma, \Sigma_+ \cup \partial\Sigma \times [0,1], \Sigma_- \cup \partial\Sigma \times [-1,0], \partial\Sigma \times \{0\}\right ).\]
\end{exemple}

\begin{definition}\label{def:ahyperbolic}
	We call a taut sutured manifold $(M,R_+,R_-,\gamma)$  \emph{ahyperbolic}  if there is a disjoint union $\mathcal{C}$ of properly embedded incompressible tori $T_1,\ldots, T_n$ and annuli $A_1,\ldots, A_k$ in $M$ such that each component $M'$ of $M\doublesms \mathcal{C}$ is, as a pair of spaces $(M',M'\cap R_- )$, homeomorphic to one of the following three simple types:
		\bnm
		\item[(a)] $(N,F)$ where $N$ is a Seifert fibered space and $F$ is a union of boundary tori and of $\pi_1$-injective annuli lying in the boundary,
		\item[(b)] $(V, C)$ where $V$ is a solid torus and $C$ is a collection of essential annuli in the boundary of $V$ (here essential means $\pi_1$-injective in $V$),
		\item[(c)] $(S\times I, S\times\left\{-1\right\})$, where $S$ is a surface (possibly with boundary) which is neither a disk nor a sphere. 
		\enm 
\end{definition}

\begin{definition}
	A Thurston norm minimizing surface $\Sigma$ in a $3$-manifold $N$ with empty or toroidal boundary is called \emph{ahyperbolic} if $N\doublesms\Sigma$, viewed as a sutured manifold, is ahyperbolic.
\end{definition}

\begin{proposition}\label{prop:torsionahyperbolic}
	If $(M,R_+,R_-,\gamma)$ is ahyperbolic, then
	\[ \tautwo(M,R_-)= 1. \]
	In particular if $\Sigma$ in $N$ is a ahyperbolic surface, then
	\[\tautwo(N\doublesms\Sigma, \Sigma_-) = 1. \]
\end{proposition}
\begin{proof}
	This proposition follows from Lemma~\ref{lem:tautwo-basics} in the following way. We first look at the three pairs (a), (b) and (c) in the definition of ahyperbolic.
	For a pair $(N,F)$ of type (a) we have $\tautwo(N,F)=1$ by Lemma~\ref{lem:tautwo-basics} (2), (3) and (4) and the fact that every Seifert fibered space is finitely covered by an $S^1$-bundle \cite[Flowchart 1]{AFW15}.

	For a pair $(V,C)$ of type (b) we have $\tautwo(V,C)=1$ by Lemma~\ref{lem:tautwo-basics} (3) and (4).	
	For a pair $(S\times I, S_-)$ of type (c) we have $\tautwo(S\times I, S_-)$ by Lemma~\ref{lem:tautwo-basics} (1).

	The proposition follows by decomposing $(M,R_+)$ along the annuli and tori of the definition of ahyperbolicity and successively applying Lemma~\ref{lem:tautwo-basics} (5). All terms appearing on the right in the equation of Lemma~\ref{lem:tautwo-basics} (5) are equal to $1$.
	Note that implicitly we used the incompressibility of the annuli and tori, without explicitly mentioning it.
\end{proof}

In the remainder of this section we will give examples of ahyperbolic surfaces.

\begin{proposition}
	Let $N$ be a graph manifold, then every Thurston norm minimizing surface $\Sigma$ in $N$ is ahyperbolic and hence
	\[ \tautwo(N\doublesms \Sigma,\Sigma_-)\,\,=\,\,1.\]
\end{proposition}

\begin{proof}
First we assume that $N$ is a Seifert fibered space. By \cite[Theorem~VI.34]{Ja80}  every Thurston norm minimizing surface is one of the following:
\begin{enumerate}
\item it is  either a disjoint union of fibers of a fibration over $S^1$ or
\item it is  a disjoint union of tori and annuli, each of which is saturated in the Seifert fibration (i.e. an union of fibers).
\end{enumerate}
In the first case $N\doublesms \Sigma$ consists of products and hence are of type (c). 
In the other case $N\doublesms \Sigma$ is 
a Seifert-fibered space with boundary, so every boundary fiber is $\pi_1$-injective, thus
$N\doublesms \Sigma$ is
of type (a).
 This proves the proposition for Seifert fibered spaces.

Now let $N$ be a graph manifold. Let $\mathcal{T}$ be the collection of JSJ-tori. As shown in the proof of \cite[Proposition~3.5]{EN85} we can assume that for any $N'\in\mathcal{J}(N)$ the intersection $\Sigma\cap N'$ is a Thurston norm minimizing surface in $N'$. Note that by definition of graph manifolds every $N'$ is Seifert fibered. 
The JSJ-tori in $N$ give rise to a collection of annuli and tori in $N\doublesms \Sigma$ satisfying the condition of Definition~\ref{def:ahyperbolic} so that the proposition follows from the first case.
\end{proof}

In \cite{AD15} Agol and Dunfield introduce the notation of a libroid knot. This is related to our concept in the following way. 
If a knot $K$ is libroid, then by definition, there exists  an $n\in \N$ such that $n\cdot\phi_K\in H^1(E_K; \Z)$ is represented by a surface  $\Sigma$ which is ahyperbolic.
 Recall from Lemma~\ref{lem:zero-phi}~(3) that $(E_K,n\cdot\phi_K)$ and $(E_K,\phi_K)$ yield the same leading coefficient.

In \cite[Section 6]{AD15} it is proven that the class of libroid knots contains all $2$-bridge knots. We can now combine this fact with Theorem \ref{main-theorem}, Lemma~\ref{lem:zero-phi}~(3) and Proposition \ref{prop:torsionahyperbolic} to obtain:

\begin{corollary}\label{cor:nonfiberedhyperbolic}
	There exist infinitely many non-fibered hyperbolic knots in $S^3$ such that 
	\[ C(E_K,\phi_K)\,\,=\,\,1.\]
\end{corollary}

\section{Examples from the Borromean rings}\label{section:examples}

We now compute the leading coefficients for the Borromean rings, and as a consequence for several infinite families of links as well.

\begin{remark}
Note that the Borromean rings of Figure \ref{fig boro} and the $m$-Whitehead links of Figure \ref{fig m whitehead} all have positive volume, and thus calculating their leading coefficients is \textit{a priori} nontrivial. One way among many to check that the aforementioned links have positive volume is to remark that the exterior of each of these links admits a Dehn filling that yields the exterior of a twist knot with at least $4$ crossings; then recall that such a twist knot
is not a torus knot (as can be seen by looking at the Alexander polynomial),
hence it is hyperbolic by \cite[Corollary 2]{Men84}. Finally use the fact due to Thurston \cite[Theorem 6.5.6]{Th79} that volume decreases under Dehn filling.
\end{remark}

\begin{prop} \label{prop boro}
Let $B = B_1 \cup B_2 \cup B_3$ be the Borromean rings, $E_B = S^3 \setminus \nu B$ the corresponding link exterior and let 
$\phi \in H^1(E_B;\Q) \cong \Q^3$. 
If $ \phi \neq 0$ then 
$C(E_B,\phi)=1$.
\end{prop}

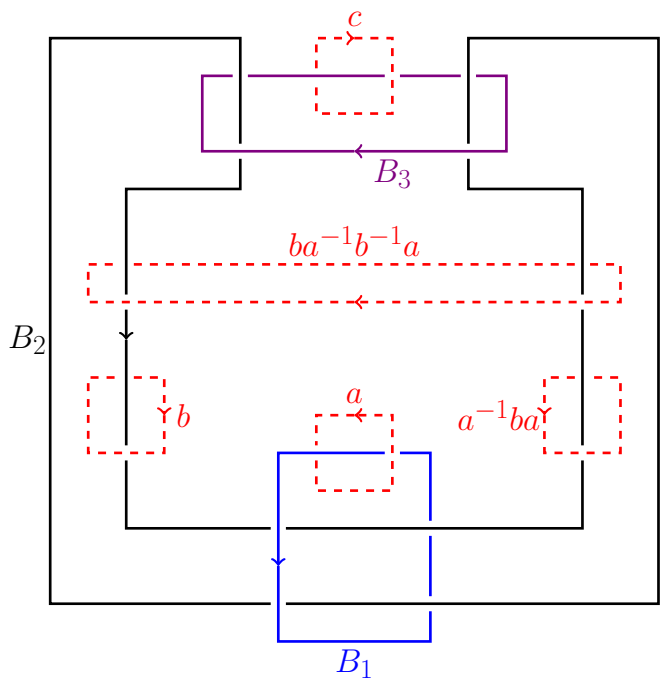
\begin{figure}[h]
\centering
\begin{tikzpicture}[every path/.style={string ,black} , every node/.style={transform shape , knot crossing , inner sep=1.5 pt } ]

\begin{scope}[scale=0.5]

\draw[color=violet] (1+0.2,5) -- (3-0.2,5);
\draw[color=violet][->] (3+0.2,5) -- (4,5) -- (4,3) -- (0,3);
\draw[color=violet] (0,3) -- (-4,3) -- (-4,5) -- (-3-0.2,5);
\draw[color=violet] (-3+0.2,5) -- (1-0.2,5);
\draw[color=violet]  (1,2.4) node {\huge $B_3$} ;

\draw (3,3-0.2) -- (3,2) -- (6,2) -- (6,-1+0.2);
\draw (6,-1-0.2) -- (6,-5+0.2);
\draw (6,-5-0.2) -- (6,-7) -- (-2+0.2,-7);
\draw (-2-0.2,-7) -- (-6,-7) -- (-6,-5-0.2);
\draw (-6,-5+0.2) -- (-6,-2);
\draw[->] (-6,-1-0.2) -- (-6,-2);
\draw (-6,-1+0.2) -- (-6,2) -- (-3,2) -- (-3,3-0.2);
\draw (-3,3+0.2) -- (-3,6) -- (-8,6) -- (-8,-9) -- (-2-0.2,-9);
\draw (-2+0.2,-9) -- (8,-9) -- (8,6) -- (3,6) -- (3,3+0.2);
\draw  (-8.6,-2) node {\huge $B_2$} ;

\draw[color=blue] (-2,-8) -- (-2,-10) -- (2,-10) -- (2,-9-0.2);
\draw[color=blue] (2,-9+0.2) -- (2,-7-0.2);
\draw[color=blue] (2,-7+0.2) -- (2,-5) -- (1+0.2,-5);
\draw[color=blue][->] (1-0.2,-5) -- (-2,-5) -- (-2,-8);
\draw[color=blue]  (0,-10.6) node {\huge $B_1$} ;

\draw[style=dashed][color=red][->] (-1,-5-0.2) -- (-1,-6) -- (1,-6) -- (1,-4) --(0,-4);
\draw[style=dashed][color=red] (0,-4) -- (-1,-4) -- (-1,-5+0.2);
\draw[color=red]  (0,-3.5) node {\huge $a$} ;

\draw[style=dashed][color=red][->] (-6+0.2,-3) -- (-5,-3) --(-5,-4);
\draw[style=dashed][color=red] (-5,-4) -- (-5,-5) -- (-7,-5) -- (-7,-3) -- (-6-0.2,-3);
\draw[color=red]  (-5+0.5,-4) node {\huge $b$} ;

\draw[style=dashed][color=red][->] (6-0.2,-3) -- (5,-3) --(5,-4);
\draw[style=dashed][color=red] (5,-4) -- (5,-5) -- (7,-5) -- (7,-3) -- (6+0.2,-3);
\draw[color=red]  (3.8,-4) node {\huge $a^{-1}ba$} ;

\draw[style=dashed][color=red][->] (6+0.2,0) -- (7,0) -- (7,-1) -- (0,-1);
\draw[style=dashed][color=red] (0,-1) -- (-7,-1) -- (-7,0) -- (-6-0.2,0);
\draw[style=dashed][color=red] (-6+0.2,0) -- (6-0.2,0);
\draw[color=red] (0,0.6) node {\huge $ba^{-1}b^{-1}a$} ;

\draw[style=dashed][color=red][->] (-1,5+0.2) -- (-1,6) -- (0,6);
\draw[style=dashed][color=red] (0,6) -- (1,6) -- (1,4) -- (-1,4) -- (-1,5-0.2);
\draw[color=red]  (0,6.5) node {\huge $c$} ;

\end{scope}

\end{tikzpicture}
\caption{The Borromean rings $B = B_1 \cup B_2 \cup B_3$} \label{fig boro}
\end{figure}

\begin{proof}
The group $G$ of the link $B$ admits the presentation
$\langle a,b,c | r, s\rangle$
where $a,b,c$ are respective meridians of $B_1, B_2, B_3$ and (with the notation $[x,y] := x y x^{-1} y^{-1}$) the relators are $r=[a,[c,b^{-1}]]$ and $s=[b,[a,c^{-1}]]$. The previous group presentation can be obtained with the Wirtinger process from the diagram  in Figure  \ref{fig boro}, where the elements of $G$ are drawn with red dashed lines. Note that among the three relators $r,s$ and $[c,[b,a^{-1}]]$, any one is a consequence of the two others, and each represents the commutation relation between a meridian and a longitude of a component $B_i$.

Let us assume $\phi \neq 0$. Then without loss of generality we can assume that $\phi(b) \neq 0$ up to reordering components, then that $\phi(b) < 0$ up to changing orientations, and finally that $\phi(b)=-1$ up to multiplying $\phi$ by a positive rational number, which does not change the value of $C(E_B,\phi)$ by Lemma \ref{lem:zero-phi} (3). Denote $\alpha=\phi(a)$ and $\gamma=\phi(c)$, two rational numbers.

The $L^2$-Alexander torsion of $(E_B,\phi)$ admits a representative of the form
\begin{align*}
\tautwo (E_B,\phi) (t) &= \dfrac{\det_G
\begin{pmatrix}
(\id - t^\gamma R_c)  t R_{cb^{-1}c^{-1}} (t^\alpha R_a - \id) 
&
 \id - R_{ac^{-1}a^{-1}c} 
 \\
(\id - t R_{cb^{-1}c^{-1}}) (t^\alpha R_a - \id) 
&
(t^\alpha R_a - \id)  \frac{1}{t^{\gamma}} R_{ac^{-1}a^{-1}} 
\left (\id - \frac{1}{t}R_{b}\right )
\end{pmatrix}
}{\det_G\left (t^\alpha R_a - \id \right )}\\
&=\det{\hspace*{-0.05cm}}_G \begin{pmatrix}
(\id - t^\gamma R_c)  t R_{cb^{-1}c^{-1}}
&
 \id - R_{ac^{-1}a^{-1}c} 
 \\
\id - t R_{cb^{-1}c^{-1}}  
&
(t^\alpha R_a - \id) \frac{1}{t^{\gamma}} R_{ac^{-1}a^{-1}} 
\left (\id - \frac{1}{t}R_{b}\right )
\end{pmatrix},
\end{align*}
where the first equality comes from Fox calculus and \cite[Proposition 3.3]{BAC17} and the second one from elementary operations on matrices (right product on the first column), see \cite[Section 3.2]{Lu02}.

Let us denote $A_t = (\id - t^\gamma R_c)  (t R_{cb^{-1}c^{-1}})$, $B_t= \id - R_{ac^{-1}a^{-1}c}$, $C_t= \id - t R_{cb^{-1}c^{-1}}$ and $D_t = (t^\alpha R_a - \id)  \left (\frac{1}{t^{\gamma}} R_{ac^{-1}a^{-1}}\right ) \left (\id - \frac{1}{t}R_{b}\right )$.
It follows from properties of the Fuglede-Kadison determinant (see for example \cite[Theorem 3.14]{Lu02}) that for all $A,B,C,D \in \mathcal{N}(G)$ such that $C$ is invertible, one has
$\det_G \begin{pmatrix} A & B \\ C & D \end{pmatrix}
= \det_G(C) \det_G(A C^{-1} D - B)$.
If we fix $t\in (0,1)$, then $C_t= \id - t R_{cb^{-1}c^{-1}}$  is invertible with  $C_t^{-1} = \sum_{n=0}^{\infty} (t R_{cb^{-1}c^{-1}})^n$. 
It follows that
$\tautwo (E_B,\phi) (t) = \max\{1,t\} \cdot \det_G(A_t C_t^{-1} D_t - B_t)$
from the above observation and from Proposition \ref{prop:op_det} (6).
Then $S_t:=A_t C_t^{-1} D_t - B_t$ has the following form:
$$
S_t=R_{ac^{-1}a^{-1}c} - \id +
(\id - t^\gamma R_c) \left (\sum_{n=1}^{\infty} (t R_{cb^{-1}c^{-1}})^n\right ) (t^\alpha R_a - \id)  \left (\frac{1}{t^{\gamma}} R_{ac^{-1}a^{-1}}\right ) \left (\id - \frac{1}{t}R_{b}\right ).
$$

Now remark that $S_t = t^d \left (U_0 + \sum_{n=1}^{\infty} t^n U_n\right )$ where the $U_n$ are in $R_{\Z[G]}$, and the lowest degree $d\in \Q$ satisfies
$$d=\min\{0,\min\{0,\alpha\}+\min\{0,\gamma\}-\gamma\}
=\min\{0,\alpha\}+\min\{0,-\gamma\}.$$ 
We claim that for all possible values of $\alpha,\gamma$, we have that $U_0$ is injective and $\det_G(U_0)=1$.

To prove the claim, we first compute the different values of $(d,U_0)$ depending on the values of $\alpha,\gamma$:
$$\kbordermatrix{\mbox{}& \alpha>0 & \alpha=0 & \alpha<0 \\
\gamma>0 & (-\gamma,R_{ac^{-1}a^{-1}}) & (-\gamma,R_{ac^{-1}a^{-1}}-R_{bab^{-1}c^{-1}}) & (\alpha-\gamma,-R_{bab^{-1}c^{-1}}) \\
\gamma=0 & (0,R_{ac^{-1}a^{-1}}-\id) & (0,(R_{bab^{-1}}-\id)(\id-R_{ac^{-1}a^{-1}})) & (\alpha,(\id-R_{c^{-1}})R_{bab^{-1}}) \\
\gamma<0 & (0,-\id) & (0,R_{bab^{-1}}-\id) &  (\alpha,R_{bab^{-1}})
}.$$
As an example, we consider the case when $\gamma=0$ and $\alpha>0$ in more details: we get $d=0$ and from the expression of $S_t$, also
\[\ba{rl} U_0&= R_{ac^{-1}a^{-1}c} - \id +
(\id - R_c)  R_{cb^{-1}c^{-1}} R_{ac^{-1}a^{-1}} R_b
\\ &= R_{ac^{-1}a^{-1}c} - \id + R_{bac^{-1}a^{-1}cb^{-1}c^{-1}} - R_{bac^{-1}a^{-1}cb^{-1}}
\\ &= R_{ac^{-1}a^{-1}c} - \id + R_{ac^{-1}a^{-1}} - R_{ac^{-1}a^{-1}c} = R_{ac^{-1}a^{-1}} - \id,
\ea
\]
where the third equality comes from the relation $s=[b,[a,c^{-1}]]$ in $G$.

Note that the second term $U_0$ of each entry in the previous matrix is of one of the forms $\pm R_g, \pm (R_g - \id)R_h$ or $\pm (R_g - \id)(R_h-\id)$ with $g,h \in G$ not equal to the neutral element of $G$ (this can be checked by the fact that they have non trivial abelianizations in $G^{ab}\cong \Z^3$), and all these operators are injective with Fuglede-Kadison determinant $1$ by Proposition \ref{prop:op_det} (6)
(where the $t$ in the proposition is $1$ here)
 and \cite[Section 3.2]{Lu02} (notably the fact that the Fuglede-Kadison determinant is multiplicative on injective operators).
This concludes the proof of the claim.

It now follows from the claim and the upper semi-continuity of the Fuglede-Kadison determinant \cite[Section 3.7]{Lu02} that
\begin{equation}\label{equ1}
\lim_{t\to 0^+}t^{-d} \tau^{(2)}(E_B,\phi)(t) =
  \lim_{t\to 0^+}t^{-d}\det\hspace*{-0.05cm}{}_{G} (S_t) \leqslant \det\hspace*{-0.05cm}{}_G(U_0)=1,\end{equation} 
where the first equality comes from the fact that 
$\lim_{t\to 0^+} \max\{1,t\} =1$.

We can now similarly study the asymptotical behavior of the same representative function $\tau^{(2)}(E_B,\phi)(t)$, as $t \to \infty$ this time. It follows from  \cite[Section 3.2]{Lu02} that:
\begin{align*}
\tau^{(2)}(E_B,\phi)(t) 
&=\det{\hspace*{-0.05cm}}_G \begin{pmatrix}
(\id - t^\gamma R_c)  t R_{cb^{-1}c^{-1}}
&
 \id - R_{ac^{-1}a^{-1}c} 
 \\
\id - t R_{cb^{-1}c^{-1}}  
&
(t^\alpha R_a - \id) \frac{1}{t^{\gamma}} R_{ac^{-1}a^{-1}} 
\left (\id - \frac{1}{t}R_{b}\right )
\end{pmatrix} \\
&= t \cdot \det{\hspace*{-0.05cm}}_G \begin{pmatrix}
\id - t^\gamma R_c 
&
 \id - R_{ac^{-1}a^{-1}c} 
 \\
t^{-1} R_{cbc^{-1}}  - \id
&
(t^\alpha R_a - \id) \frac{1}{t^{\gamma}} R_{ac^{-1}a^{-1}} 
\left (\id - \frac{1}{t}R_{b}\right )
\end{pmatrix},
\end{align*}
and thus $\tau^{(2)}(E_B,\phi)(t)  = t \det_G(C'_t) 
\det_G(S'_t)$ where $B_t,D_t$ are as before, $A'_t =\id - t^
\gamma R_c$, $C'_t= t^{-1} R_{cbc^{-1}}  - \id$ and $S'_t=A'_t C'^{-1}_t D_t -B_t$. Similarly as before we remark that
$$S'_t = t^D \left (V_0 + \sum_{n=1}^{\infty} t^{-n} V_n\right ),$$
where 
the highest degree $D$ satisfies
 $$D=\max\{0,\max\{0,\alpha\}+\max\{0,\gamma\}-\gamma\}
 =\max\{0,\alpha\}+\max\{0,-\gamma\},
 $$
   $V_n \in R_{\Z[G]}$ for all $n\geqslant  0$ and $V_0$ is always injective with Fuglede-Kadison determinant $1$. This last claim is once again checked for every one of the nine distinct possibilities, summarised in the following table of pairs $(D,V_0)$: 
$$\kbordermatrix{\mbox{}& \alpha>0 & \alpha=0 & \alpha<0 \\
\gamma>0 & (\alpha,R_{a}) & (0,R_a - \id) & (0,-R_{ac^{-1}a^{-1}c}) \\
\gamma=0 & (\alpha,(R_c-\id)R_{ac^{-1}}) & (0,(R_{ac^{-1}a^{-1}}-\id)(\id-R_{a})) & (0,R_{ac^{-1}a^{-1}}-\id) \\
\gamma<0 & (\alpha-\gamma,-R_{ac^{-1}}) & (-\gamma,(R_a-\id)R_{ac^{-1}a^{-1}}) &  (-\gamma,R_{ac^{-1}a^{-1}})
}.$$
Hence, by the upper semi-continuity of the Fuglede-Kadison determinant \cite[Section 3.7]{Lu02} and the fact that $\det_G(C'_t) = \max\{1,t^{-1}\} \sim_{t \to \infty} 1$, we have:
\begin{equation}\label{equ2}
\lim_{t\to \infty}t^{-1-D} \tau^{(2)}(E_B,\phi)(t) =  \lim_{t\to \infty}t^{-D}\det\hspace*{-0.05cm}{}_{G} (S'_t) \leqslant \det\hspace*{-0.05cm}{}_G(V_0)=1.\end{equation}

Recall that $x_{E_B}(\phi)=|\alpha|+|\beta|+|\gamma|$ (see 
 \cite[Example 2]{Th86}),
 thus $x_{E_B}(\phi)=1+D-d$ by definition of $d$ and $D$ and the assumption that $\beta=-1$.

Now, the inequality \eqref{equ1} implies that 
$\lim_{t\to 0^+} \frac{\ln(\tau^{(2)}(E_B,\phi)(t))}{\ln(t)} \geqslant d$ and the inequality \eqref{equ2} implies that
$\lim_{t\to \infty} \frac{\ln(\tau^{(2)}(E_B,\phi)(t))}{\ln(t)} \leqslant 1+D$. Hence, since $x_{E_B}(\phi)=1+D-d$, it follows from 
Theorem \ref{thm:previous2} (3) that the two  limits of the previous sentence are respectively equal to $d$ and $1+D$.
Finally, it follows from the inequality  \eqref{equ2} and Theorem \ref{thm:previous2} (4) that 
 $C(E_B,\phi) \leqslant 1$.

We conclude using the fact that $C(E_B,\phi) \geqslant 1$ from Theorem \ref{thm:previous2} (4) (a).
\end{proof}

\begin{remark}
Recall that given a knot $K$, its Bing double $B(K)$ is the $2$-component link given by the image of the link $(B \setminus B_1) \subset E_{B_1}$ (a $2$-component sublink of the Borromean rings $B$ living in an unknot exterior) under the inclusion
$E_{B_1} \subset E_{B_1}\cup E_K \cong S^3$, where
the gluing identifies a preferred longitude of the component $B_1$ to a preferred meridian of $K$ and vice-versa.
Thus, it follows  easily from  Proposition \ref{prop boro} and Theorem~\ref{thm:previous} (2) and (3) that given any knot $K$ and any non-zero $\phi\in H^1(X_{B(K)};\Z)$ the leading coefficient of the corresponding $L^2$-Alexander torsion equals $\exp(\vol(K)/6\pi)$.
\end{remark}

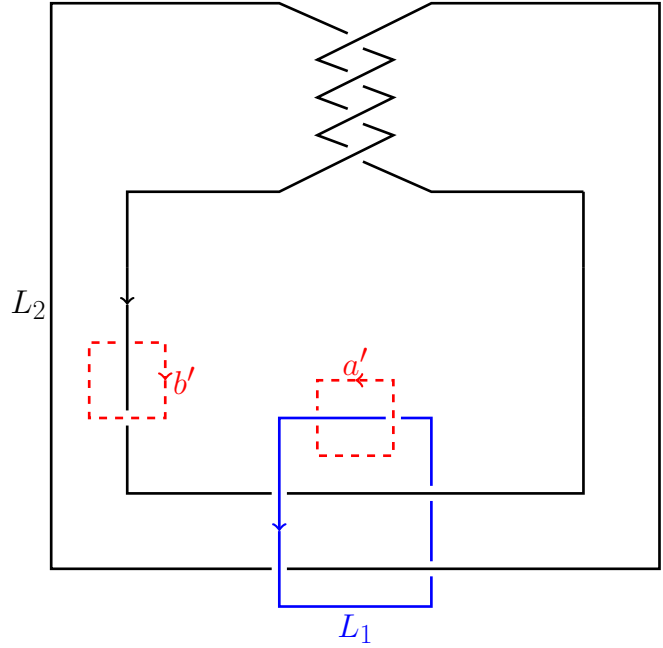
\begin{figure}[!h]
\centering
\begin{tikzpicture}[every path/.style={string ,black} , every node/.style={transform shape , knot crossing , inner sep=1.5 pt } ]

\begin{scope}[scale=0.5]

\draw (6,1) -- (6,-1);
\draw (6,-1) -- (6,-5);
\draw (6,-5) -- (6,-7) -- (-2+0.2,-7);
\draw (-2-0.2,-7) -- (-6,-7) -- (-6,-5-0.2);
\draw (-6,-5+0.2) -- (-6,-2);
\draw[->] (-6,-1) -- (-6,-2);

\draw (-6,-1) -- (-6,1) -- (-2,1) -- (1,2.5) -- (0+0.2,3-0.2);
\draw (0-0.2,2+0.2) -- (-1,2.5) -- (1,3.5)-- (0+0.2,4-0.2);
\draw (0-0.2,3+0.2) -- (-1,3.5) -- (1,4.5)-- (0+0.2,5-0.2);
\draw (0-0.2,4+0.2) -- (-1,4.5) -- (2,6) --(8,6) -- (8,-9) -- (-2+0.2,-9);
\draw (-2-0.2,-9) -- (-8,-9) -- (-8,6) -- (-2,6) -- (0-0.2,5+0.2);
\draw (0+0.2,2-0.2) -- (2,1) -- (6,1);

\draw  (-8.6,-2) node {\huge $L_2$} ;

\draw[color=blue] (-2,-8) -- (-2,-10) -- (2,-10) -- (2,-9-0.2);
\draw[color=blue] (2,-9+0.2) -- (2,-7-0.2);
\draw[color=blue] (2,-7+0.2) -- (2,-5) -- (1+0.2,-5);
\draw[color=blue][->] (1-0.2,-5) -- (-2,-5) -- (-2,-8);
\draw[color=blue]  (0,-10.6) node {\huge $L_1$} ;

\draw[style=dashed][color=red][->] (-1,-5-0.2) -- (-1,-6) -- (1,-6) -- (1,-4) --(0,-4);
\draw[style=dashed][color=red] (0,-4) -- (-1,-4) -- (-1,-5+0.2);
\draw[color=red]  (0,-3.5) node {\huge $a'$} ;

\draw[style=dashed][color=red][->] (-6+0.2,-3) -- (-5,-3) --(-5,-4);
\draw[style=dashed][color=red] (-5,-4) -- (-5,-5) -- (-7,-5) -- (-7,-3) -- (-6-0.2,-3);
\draw[color=red]  (-5+0.5,-4) node {\huge $b'$} ;

\end{scope}

\end{tikzpicture}
\caption{The $2$-component $m$-Whitehead link $W_m = L_1 \cup L_2$ (for $m=2$)} \label{fig m whitehead}
\end{figure}

Let us define the \textit{$m$-Whitehead link} $W_m$ as 
the image of two components of $B$ under 
 a $(1/m)$-Dehn filling on the boundary of a neighborhood of the third component of $B$ 
(with $m \geqslant 1$ and the $1$-Whitehead link being the classical Whitehead link), see Figure \ref{fig m whitehead}. Following Proposition \ref{prop boro} and \cite{BA16b} we can compute the leading coefficients for all $m$-Whitehead links:

\begin{prop} \label{prop white}
Let $W_m$ be the $2$-component $m$-Whitehead link, $E_{W_m} = S^3 \setminus \nu W_m$ the corresponding link exterior and let 
$\phi \in H^1(E_{W_m};\Q) \cong \Q^2$. 
If $ \phi \neq 0$ then 
$C(E_{W_m},\phi)=1$.
\end{prop}

\begin{proof}

Let $m\geqslant 1$ and let
 $\phi \in H^1(E_{W_m};\Q)\setminus\{0\}$.
  We want to prove that $C(E_{W_m},\phi)=1$. In the remainder of the proof we will use several notations from Proposition \ref{prop boro} and its proof.

Firstly, up to scaling $\phi$ and reordering the components of $W_m$ (which does not change its isotopy type) we can assume that
$\phi$ sends any meridian of the second component of $W_m$ to $-1$. See Figure \ref{fig m whitehead} for a diagram of $W_m$ in the case $m=2$.

Now consider that $W_m$ is obtained via $(1/m)$-Dehn filling on the third boundary component of $E_B$,
 as in Figures \ref{fig boro} and \ref{fig m whitehead},
and let us call $\theta_m\colon\pi_1(E_B) \twoheadrightarrow \pi_1\left (E_{W_m}\right )$ the group epimorphism induced by this Dehn filling.
It follows that $\gamma=(\phi\circ\theta_m)(c)=0$. Moreover, it follows from the previous paragraph that  $(\phi\circ\theta_m)(b)=-1$.

We state the following crucial claim.
\begin{claim}
Among the three possible operators $U_0$ (respectively  $V_0$) listed in the proof of Proposition \ref{prop boro} for  $\gamma=0$, neither  becomes the zero operator under the epimorphism $\theta_m$.
\end{claim}

We will now prove the claim. 
Under the epimorphism $\theta_m$, the three operators $U_0$ (in the case $\gamma=0$) become right multiplications by the following elements of $\Z[\pi_1(E_{W_m})]$:
$$\theta_m(ac^{-1}a^{-1})-1, \
\left (1-\theta_m(ac^{-1}a^{-1} )\right )\left (\theta_m(ba^{-1}b^{-1})-1\right ), \
\theta_m(ba^{-1}b^{-1})\left (1-\theta_m(c^{-1})\right ),$$
and the three operators $V_0$ become right multiplications by the following elements:
$$\theta_m(ac^{-1})\left (\theta_m(c)-1\right ), \
\left (1-\theta_m(a)\right )\left (\theta_m(ac^{-1}a^{-1})-1\right ), \
\theta_m(ac^{-1}a^{-1})-1.$$
As in the proof of Proposition~\ref{prop boro}  it suffices to prove that
$a':=\theta_m(a)$ and $c':=\theta_m(c)$ are non trivial in $\pi_1(E_{W_m})$.

Denoting $b':=\theta_m(b)$, it follows from the Dehn filling process and the van Kampen theorem that $c' = \left [b',a'^{-1}\right ]^{-m}$ and that $\pi_1(E_{W_m})$ is generated by $a'$ and $b'$; indeed, as seen in Figure \ref{fig boro}, $l:=[b,a^{-1}]$ is a longitude of $B_3$, 
and $(1/m)$-Dehn filling on $E_B$ on the boundary component $\partial (\nu B_3)$ 
induces a quotient by the relation $c l^m = 1$ in fundamental groups.

Now, since $a'$ still represents a meridian in $\pi_1(E_{W_m})$, it is non trivial. Moreover 
 $\pi_1(E_{W_m})$ is non abelian (since the only $2$-component link with abelian group is the Hopf link \cite{Neu61} which has different linking number from $W_m$)
and generated by $a'$ and $b'$,
 thus $\left [b',a'^{-1}\right ]$ is not trivial, and since $\pi_1(E_{W_m})$ is torsion free
(see \cite[Section 3.2 (C.3)]{AFW15}), then $c'= \left [b',a'^{-1}\right ]^{-m}$ is not trivial either, which concludes the proof of the claim.

We now recall that from \cite{DFL16} one can define a more general $L^2$-Alexander torsion $\tautwo(N,\phi \circ  \theta,\theta)(t)$, associated to a $3$-manifold $N$ of fundamental group $G$, and two group homomorphisms $\phi\colon \Gamma\to \Z$ and $\theta\colon G \to \Gamma$ such that $\Gamma$ is residually finite. The process is mostly the same as in Section \ref{section:def-l2-torsion}, except that the operators are over $\ell^2(\Gamma)$ and are linear combinations of $R_{\theta(g)}$ for $g\in G$ (instead of $\ell^2(G)$ and $R_g$).
Note that in our case, $N$ will be $E_B$, $\pi_1\left (E_{W_m}\right )$ is  residually finite \cite{Hem87,AFW15} and will play the role of the aforementioned $\Gamma$, and $\theta$ will be the $\theta_m$ defined at the beginning of the proof.

In the proof of the claim we established that $\theta_m(l) = \left [b',a'^{-1}\right ]$ is non trivial in the torsion free group $\pi_1(E_{W_m})$, thus it has infinite order. Hence we can use the Dehn surgery formula for $L^2$-Alexander torsions from \cite[Proposition 4.3]{BA16b} to conclude that
$$\tautwo(E_{W_m},\phi)(t)
 \ \dot{=} \
\dfrac{\tautwo(E_B,\phi\circ \theta_m,\theta_m)(t)}
{\max\{1,t\}^{|\phi\circ\theta_m(l)|}}
=
\tautwo(E_B,\phi\circ \theta_m,\theta_m)(t),
$$
where the first equality comes from \cite[Proposition 4.3 and Section 4.4]{BA16b} and the second equality comes from the fact that $\theta_m(l)=\left [b',a'^{-1}\right ]$ is a  commutator.

Recall that we want to prove that $C(E_{W_m},\phi)=1$. It follows from the previous paragraph that 
$C(E_{W_m},\phi)$ is equal to the leading coefficient of $\tautwo(E_B,\phi\circ \theta_m,\theta_m)(t)$. Thus it suffices to prove that the leading coefficient of $\tautwo(E_B,\phi\circ \theta_m,\theta_m)(t)$  is equal to $1$.
Note that the previous statement is a variation of Proposition \ref{prop boro}, and we will prove it in a similar way.

We can extend  Proposition \ref{prop boro} and the bulk of its proof to the  $L^2$-Alexander torsion  $\tautwo(E_B,\phi\circ \theta_m,\theta_m)(t)$ instead of 
$\tautwo(E_B,\phi)(t)$, since all $L^2$-torsions and limits are well defined thanks to \cite[Proposition 4.3]{BA16b}. The only potential problem comes from the operators $U_0$ and $V_0$, which might become zero under the epimorphism $\theta_m$. Fortunately this is not the case, thanks to the previous claim.
 \end{proof}

In particular for $m=1$, it follows from Theorem~\ref{thm:previous} (1), (2) and (3) and Lemma \ref{lem:zero-phi} (2) that we can compute the leading coefficient of Whitehead doubles:

\begin{prop}
Let $K$ be a knot and $W(K)$ its untwisted Whitehead double. Then
$$C(E_{W(K)},\phi_{W(K)}) = \exp\left (\dfrac{\vol(E_K)}{6 \pi}\right ).$$
\end{prop}

One such example is $W(4_1)$, the Whitehead double of the figure-eight knot.

\begin{proof}
Let $K$ be a knot and $W(K)$ denote its untwisted Whitehead double. Thus we have $E_{W(K)} = E_{K} \cup E_{W_1}$, where $W_1$ is the Whitehead link. The gluing is such that $\phi_{W(K)}$ restricts to $0$ on $E_K$ and to a non-zero class $\phi_1$ on $E_{W_1}$. It thus follows from Theorem~\ref{thm:previous} (2) and (3) that 
$$C(E_{W(K)},\phi_{W(K)}) = C(E_K,0) \cdot  C(E_{W_1},\phi_1).$$
Now it follows from Proposition \ref{prop white} that $C(E_{W_1},\phi_1)=1$ and from Theorem~\ref{thm:previous} (1) and Lemma \ref{lem:zero-phi} (2) that 
$C(E_K,0)=\exp\left (\dfrac{\vol(E_K)}{6 \pi}\right )$, which concludes the proof.
\end{proof}

As an immediate consequence we obtain:

\begin{corollary}\label{cor:>1}
The set of leading coefficients $C(E_K,\phi_K)$ (where the index $K$ runs over the set of all knots) is infinite.

Furthermore, this set of leading coefficients contains a subset which is bijective to the set of hyperbolic volumes $\vol(E_K)$ of hyperbolic knots.
\end{corollary}

In conclusion, let us recap in the following remark the class of links for which we can compute the leading coefficient explicitly (at the time of writing).

\begin{remark}
We can explicitly compute all leading coefficients for the class of links containing 
\begin{itemize}
\item fibered knots,
\item libroid knots $($in particular $2$-bridge knots$)$,
\item  the Borromean rings and all the links obtained from them via Dehn fillings on their exterior,
\end{itemize}
and stable by 
\begin{itemize}
	\item connected sum,
	\item cabling,
	\item Bing doubling,
	\item any $m$-Whitehead doubling.
\end{itemize}
\end{remark}

\end{document}